\documentclass[11pt]{article}
\usepackage{mathrsfs}
\usepackage[centertags]{amsmath}
\usepackage{amsfonts}
\usepackage{amssymb}
\usepackage{amsthm}
\usepackage{cases}
\usepackage{indentfirst}
\usepackage{graphicx}
\usepackage{color}
\usepackage{epsfig}
\usepackage{float}
\usepackage{subfigure}
\usepackage[figuresleft]{rotating}

\date{}

\definecolor{c20}{rgb}{0.,0.7,0.}
\definecolor{c30}{rgb}{0.,0.,1.}
\definecolor{c40}{rgb}{1,0.1,0.7}
\definecolor{c50}{rgb}{1,0,0}
\definecolor{c60}{rgb}{1,0.9,0.1}

\newtheorem{theorem}{Theorem}[section]

\newtheorem{lemma}{Lemma}[section]

\newtheorem{proposition}{Proposition}[section]
\newtheorem{remark}{Remark}[section]

\numberwithin{equation}{section}

\def\P{\operatorname*{P}}
\def\ST{\operatorname*{ST}}
\def\MSTD{\operatorname*{MSTD}}

\def\sign{\operatorname*{sign}}
\def\R{\operatorname*{\mathbb{R}}}

\allowdisplaybreaks[4]

\begin{document}
\title{\textbf{Higher-order expansions of extremes from mixed skew-$t$ distribution}}
\author{$^{a}$Jingyao Hou\quad $^{b}$Xin Liao\thanks{Corresponding author. Email: liaoxin2010@163.com}\quad $^{a}$Zuoxiang Peng \\
{\small $^{a}$School of Mathematics and Statistics, Southwest University, Chongqing, 400715, China} \\
{\small $^{b}$Business School, University of Shanghai for Science and Technology, Shanghai, 200093, China}}

\maketitle
\begin{quote}
{\bf Abstract.}~~In this paper, we study the asymptotic behaviors of the extreme of mixed skew-t distribution. We considered limits on distribution and density of maximum of mixed skew-t distribution under linear and power normalization, and further derived their higher-order expansions, respectively. Examples are given to support our findings.

\end{quote}

\section{Introduction}
\label{sec1}
The skew-$t$ distribution due to Azzalini and Capitanio (2003) can be defined as follows. Let $Y$ and $Z$ be two independent random variables with $Y$ following $\chi^2(v)$ with degree of freedom $v$ and $Z$ being a skew-normal random variable with probability density function (pdf)
$$f_{Z}(x)=2\phi(x)\Phi(\beta x)$$
for $x\in \R$ and parameter $\beta\in \R$, where $\phi(\cdot)$ denote the standard normal pdf and $\Phi(\cdot)$ denotes the standard normal cumulative distribution function (cdf).
Define $X=Y/\sqrt{Z/v}$, then $X$ is said to have skew-$t$ distribution, written as $X\sim \ST_v(\beta)$
for short. Azzalini and Capitanio (2003)  showed that the pdf of $\ST_v(\beta)$ is
\begin{eqnarray}\label{eq1.1}
f(x)=2t_v(x)T_{v+1}\left(\beta x\sqrt{\frac{v+1}{x^2+v}}\right),
\end{eqnarray}
where $t_v(\cdot)$ is the pdf of the standard Student's $t$ distribution with degree of freedom $v$, and $T_{v+1}(\cdot)$
is the cdf of the standard Student's $t$ distribution with degree of freedom $v+1$. Note that
\begin{eqnarray}\label{eq1.2}
t_v(x)=C_v\left(1+\frac{x^2}{v}\right)^{-\frac{v+1}{2}},
\end{eqnarray}
where $x\in \R$ and $C_v=\frac{\Gamma\left(\frac{v+1}{2}\right)}{\Gamma\left(\frac{v}{2}\right)\sqrt{v\pi}}$. It follows from \eqref{eq1.1} and \eqref{eq1.2} that $f\in RV_{-\alpha-1}$ implying $F\in D(\Phi_{\alpha})$, i.e, there exist norming constant $a_{n}>0$ such that
\begin{equation}\label{eq1.3}
\lim_{n\to\infty}F^{n}(a_{n}x)=\Phi_{\alpha}(x)=\exp(-x^{-\alpha}), \;\; x>0,
\end{equation}
where $F$ is the cdf of skew-t distribution. Here, $g\in RV_{\beta}$ means $\lim_{t\to\infty}g(tx)/g(t)=x^{\beta}$ for  $x>0$. Further extremal properties such as distributional expansions of extremes from skew-t distribution, see Peng et al. (2016).

Nowadays finite mixed distribution has received widespread research.
Peel and Mclachlan (2000) considered
a robust approach by modelling atypical observations by a mixture of Student-t distributions. Frigessi et al.(2002) proposed a dynamic mixture
approach to estimate the severity distributions. Dellaportas and Papageorgiou (2006) presented full Bayesian analysis of finite mixtures of multivariate normals with unknown number of components.
Cabral et al. (2008) gave a Bayesian approach for modeling heterogeneous data and estimated
multimodal densities using mixtures of skew student-t-normal distributions.
Sattayatham and Talangtam (2012) modeled
motor insurance claims data from Thailand by using a mixture of log-normal distributions.
Lin et al. (2007) proposed a robust mixture framework based on the skew-t distribution,
and provided EM-type algorithms for iteratively computing maximum likelihood estimates.
Fr\"{u}hwrith-Schnatter and Pyne (2010) investigated Bayesian inference for finite mixtures of skew-normal and skew-t distributions, and applied them to modelling non-Gaussian cell populations.
Ho and Lin (2010) considered a robust linear mixed skew-t models with application to schizophrenia data. A two-component skew-t mixture model is used to analyze freeway speed
data characteristics by Zou and Zhang (2011). For more
 studies related to the mixed multivariate
skew-t distribution, we refer to Lin (2010), Vrbik and McNicholas (2012), Lee and McLachlan (2013, 2014).

The objective of this paper is to study the asymptotic behaviors of extremes of finite mixed skew-t distribution (shortened by $\MSTD$) under linear and power normalization, respectively. The finite $\MSTD$ is defined as follows. Let $X_1,X_2,\cdots,X_r$ be independent random variables with $X_i\sim \ST_{v_i}(\beta_i)$, where $v_i>0$ and $\beta_i\in \R$ for $i=1,2,\cdots,r.$ Without loss of generality we suppose that
$v_1<v_2<\cdots<v_r$. Define a new random variable $T$ by
\begin{eqnarray}\label{eq1.3}
T=\left\{
\begin{aligned}
&X_1, & \quad \P(T=X_1)=p_1& \\
&X_2, & \quad \P(T=X_2)=p_2& \\
&\vdots & \vdots \quad & \\
&X_r, & \quad \P(T=X_r)=p_r&,
\end{aligned}
\right.
\end{eqnarray}
where $p_i>0$($1\le i\le r$) satisfying $\sum_{i=1}^rp_i=1$. Then $T$ is said to have finite $\MSTD$ with $r$ components.
Denoting by $F_{v,\beta}(x)$ the cdf of a random variable $X\sim \ST_{v}(\beta)$ and $F(x)$ the cdf of the random variable $T$, we can easily get
\begin{eqnarray}\label{eq1.4}
F(x)=p_1F_{v_1,\beta_1}(x)+p_2F_{v_2,\beta_2}(x)+\cdots+p_rF_{v_r,\beta_r}(x).
\end{eqnarray}

Contents of this paper are organized as follows. In Section \ref{sec2}, expansion of the distributional tail of $\MSTD$  shows that extreme value distribution from $\MSTD$ sample is Fr\'{e}chet distribution under given linear normalization, implying that under power normalization the extreme value distribution from $\MSTD$ is $\Phi_{1}(x)$.  Higher-order expansions of the cdf and the pdf of extremes from $\MSTD$ are given in Section \ref{sec3}. Numerate analysis provided in Section \ref{sec4} compare the asymptotic behaviors under different normalization.

\section{Preliminaries}\label{sec2}
In this section, we provide some primary results related to $\ST$ and $\MSTD$. The first one is from Peng et al. (2016).
\begin{lemma}
\label{Lemma2.1}
Let $F_{v,\beta}(x)$ denote the cdf of the skew-$t$ distribution. For large $x$, we have
\begin{eqnarray}\label{eq2.1}
1-F_{v,\beta}(x)=2C_vv^\frac{v-1}{2}T_{v+1}(\beta\sqrt{v+1})x^{-v}\left(1+A_1x^{-2}+A_2x^{-4}+O(x^{-6})\right),
\end{eqnarray}
where
\[A_1=-\frac{C_{v+1}v^2\sqrt{v+1}\beta(1+\beta^2)^{-\frac{v+2}{2}}}{2T_{v+1}(\beta\sqrt{v+1})(v+2)}-\frac{v^2(v+1)}{2(v+2)}\]
and
\begin{eqnarray*}
A_2&=&
\frac{C_{v+1}\sqrt{v+1}\beta(1+\beta^2)^{-\frac{v+2}{2}}}{T_{v+1}(\beta\sqrt{v+1})}
\left[\frac{3v^2}{8}+\frac{v^3(v-1)}{4(v+2)}-\frac{3v^2}{(v+2)(v+4)}-\frac{(v+2)\beta^2v^3}{8(v+4)(1+\beta^2)} \right]\\
& &+\frac{v^2(v^2-1)}{8}+\frac{3v^2}{(v+2)(v+4)}+\frac{v^2(v-1)}{2(v+2)}.
\end{eqnarray*}
\end{lemma}
For the distibutional tail behavior of $\MSTD$, we have the following results.
\begin{lemma}
\label{Lemma2.2}\label{lemma2.2}
Let F(x) be the cdf of random variable $T$, for large $x$ we have
\begin{eqnarray*}
1-F(x)&=&2p_1C_{v_1}v_{1}^{\frac{v_1-1}{2}}T_{v_1+1}(\beta_1\sqrt{v_1+1})x^{-v_1}(1+A_1x^{-2}+A_2x^{-4}+A_3x^{-(v_2-v_1)}\\
&&+A_4x^{-(v_2-v_1)-2}+A_5x^{-(v_3-v_1)}+O(x^{-\eta}))
\end{eqnarray*}
where $\eta=\min\{6,v_4-v_1,v_2-v_1+4,v_3-v_1+2\}$ and
\begin{eqnarray*}
&&A_1=-\frac{C_{v_1+1}v_{1}^{2}\sqrt{v_1+1}\beta_1(1+\beta_{1}^{2})^{-\frac{v_1+2}{2}}}{2T_{v_1+1}(\beta_1\sqrt{v_1+1})(v_1+2)}-\frac{v_{1}^{2}(v_1+1)}
{2(v_1+2)};\\
&&A_2=\frac{C_{v_1+1}\sqrt{v_1+1}\beta_1(1+\beta_1^2)^{-\frac{v_1+2}{2}}}{T_{v_1+1}(\beta_1\sqrt{v_1+1})}
\left[\frac{v_1^3(v_1-1)}{4(v_1+2)} +\frac{3v_{1}^2}{8}-\frac{3v_{1}^2}{(v_1+2)(v_1+4)}-\frac{(v_1+2)\beta_1^2v_1^3}{8(v_1+4)(1+\beta_1^2)} \right]\\
&&+\frac{v_1^2(v_1^2-1)}{8}+\frac{3v_1^2}{(v_1+2)(v_1+4)}+\frac{v_1^2(v_1-1)}{2(v_1+2)};\\
&&A_3=\frac{p_2C_{v_2}v_2^{\frac{v_2-1}{2}}T_{v_2+1}(\beta_2\sqrt{v_2+1})}{p_1C_{v_1}v_1^{\frac{v_1-1}{2}}T_{v_1+1}(\beta_1\sqrt{v_1+1})};\\
&&A_4=-\frac{p_2C_{v_2}v_2^{\frac{v_2-1}{2}}T_{v_2+1}(\beta_2\sqrt{v_2+1})}{p_1C_{v_1}v_1^{\frac{v_1-1}{2}}T_{v_1+1}(\beta_1\sqrt{v_1+1})}
\left(\frac{C_{v_2+1}v_2^{2}\sqrt{v_2+1}\beta_2(1+\beta_2^2)^{-\frac{v_2+2}{2}}}{2(v_2+2)T_{v_2+1}(\beta_2\sqrt{v_2+1})}+\frac{v_2^2(v_2+1)}{2(v_2+2)}\right);\\
&&A_5=\frac{p_3C_{v_3}v_3^{\frac{v_3-1}{2}}T_{v_3+1}(\beta_3\sqrt{v_3+1})}{p_1C_{v_1}v_1^{\frac{v_1-1}{2}}
T_{v_1+1}(\beta_1\sqrt{v_1+1})}.
\end{eqnarray*}
\end{lemma}
\begin{proof}
To simplify the proof, we first define two functions $A_1(\cdot,\cdot)$ and $A_2(\cdot,\cdot)$ as following.
For $i=1,\cdots,r,$ define
\begin{eqnarray*}
A_1(v_i,\beta_i)=-\frac{C_{v_i+1}v_i^2\sqrt{v_i+1}\beta_i(1+\beta_i^2)^{-\frac{v_i+2}{2}}}{2(v_i+2)T_{v_i+1}(\beta_i\sqrt{v_i+1})}
-\frac{v_i^2(v_i+1)}{2(v_i+2)}
\end{eqnarray*}
and
\begin{eqnarray*}
A_2(v_i,\beta_i)&=&
\frac{C_{v_i+1}\sqrt{v_i+1}\beta(1+\beta_i^2)^{-\frac{v_i+2}{2}}}{T_{v_i+1}(\beta\sqrt{v_i+1})}\\
&&\times \left[\frac{v_i^3(v_i-1)}{4(v_i+2)} +\frac{3v_{i}^2}{8}-\frac{3v_{i}^2}{(v_i+2)(v_i+4)}-\frac{(v_i+2)\beta_i^2v_i^3}{8(v_i+4)(1+\beta_i^2)} \right]\\
&&+\frac{v_i^2(v_i^2-1)}{8}+\frac{3v_i^2}{(v_i+2)(v_i+4)}+\frac{v_i^2(v_i-1)}{2(v_i+2)}.
\end{eqnarray*}
It follows from \eqref{eq2.1} that
\begin{eqnarray*}
1-F(x)&=&\sum_{i=1}^r p_i\big(1-F_{v_i,\beta_i}(x)\big)\\
&=&\sum_{i=1}^r2p_{i}C_{v_i}v_i^{\frac{v_i-1}{2}}T_{v_i+1}(\beta_i\sqrt{v_i+1})x^{-v_i}\big(1+A_1(v_i,\beta_i)x^{-2}+A_2(v_i,\beta_i)x^{-4}+O(x^{-6})\big)\\
&=&2p_{1}C_{v_1}v_1^{\frac{v_1-1}{2}}T_{v_1+1}(\beta_1\sqrt{v_1+1})x^{-v_1}\left(1+A_1x^{-2}+A_2x^{-4}+A_3x^{-(v_2-v_1)}\right.\\
&&\left.+A_4(v_2,\beta_2)x^{-(v_2-v_1)-2}+A_5x^{-(v_3-v_1)}+O(x^{-\eta^{'}})\right)\\
&&+\sum_{i=4}^r2p_iC_{v_i}v_i^{\frac{v_i-1}{2}}
T_{v_i+1}(\beta_i\sqrt{v_i+1})x^{-v_i}\big(1+A_1(v_i,\beta_i)x^{-2}+A_2(v_i,\beta_i)x^{-4}+O(x^{-6})\big),
\end{eqnarray*}
where $\eta^{'}=\min\{6,v_2-v_1+4,v_3-v_1+2\}$. The proof is complete by noting that the last trem is the order of $x^{-v_4}$.
\end{proof}

\begin{proposition}
\label{prop2.1}
Let $\{T_n,n\le 1\}$ be a sequence of independent and identical distribution random variables with marginal cdf $F(x)$. Let $M_n=\max_{1\le k \le n}\{T_k\}$ denote the partial maximum. Then we have
$$\lim_{n\to\infty}\P(M_{n}\le a_{n}x)=\lim_{n\to \infty}F^n(a_n x)=\exp(-x^{-v_1}),$$
where
\begin{eqnarray}\label{eq2.2}
a_n=(2p_1C_{v_1}T_{v_1+1}(\eta_1\sqrt{v_1+1}))^{\frac{1}{v_1}}v_1^{\frac{v_1-1}{2v_1}}n^{\frac{1}{v_1}}.
\end{eqnarray}
\end{proposition}
\begin{proof}
From Lemma \ref{Lemma2.2}, for large $x$ we have
\begin{eqnarray}\label{eq2.3}
1-F(x)\sim 2p_1C_{v_1}v_1^{\frac{v_1-1}{2}}T_{v_1+1}(\beta_1\sqrt{v_1+1})x^{-v_1},
\end{eqnarray}
which implies
$$\lim_{t\to \infty}\frac{1-F(tx)}{1-F(t)}=x^{-v_1}$$
for $x>0$. Thus by Proposition 1.11 in Resnick(1987), we have $F(x)\in D(\Phi_{v_1})$. The remainder is to compute the norming constant $a_n$.

Since the cdf $F(x)$ is continuous for $x\in \R$, there exist $t_n$ for each integer $n\ge 2$ such that
$$n(1-F(t_n))=1.$$
Hence by \eqref{eq2.3} we have
$$2np_1C_{v_1}v_1^{\frac{v_1-1}{2}}T_{v_1+1}(\beta_1\sqrt{v_1+1})t_n^{-v_1}\to 1$$
as $n\to \infty$, implying
$$t_n=n^{\frac{1}{v_1}}(2p_1C_{v_1}T_{v_1+1}(\beta_1\sqrt{v_1+1}))^{\frac{1}{v_1}}v_1^{\frac{v_1-1}{2v_1}}(1+o(1)).$$
Let $a_n=n^{\frac{1}{v_1}}(2p_1C_{v_1}T_{v_1+1}(\beta_1\sqrt{v_1+1}))^{\frac{1}{v_1}}v_1^{\frac{v_1-1}{2v_1}}$, and the result follows by  Khintchine Theorem in Leadbetter et al. (1983).
\end{proof}

To end this section, we provide the limiting distribution of maximum of MSTD
under power normalization. For the extreme value distributions under power normalization, we refer the reader to the original work of Pancheva (1985). By Theorem 3.1 in Mohan and Ravi (1993) and Proposition \ref{prop2.1}, we have the following result.

\begin{proposition}\label{prop2.2}
Let $F$ be the cdf of $\MSTD$. With $\alpha_n=a_n$ and $\beta_n=1/v_1$ we have
\begin{equation}
\lim_{n\to \infty}F^n(\alpha_n|x|^{\beta_{n}}\sign(x))=\Phi_1(x).
\end{equation}
\end{proposition}

\section{Higher-order expansions of extremes under different normalization}\label{sec3}
In this section, we consider the higher-order expansions of the cdf and the pdf of extremes under linear and power normalization, respectively. To simplify our result, we first introduce two indicative functions $I(\cdot,\cdot,\cdot)$ and $J(\cdot,\cdot)$ such that
\begin{eqnarray*}
I(v_1\in B_1,v_2 \in B_2,v_3\in B_3)=\left\{
\begin{aligned}
&1, &\quad \mbox{if}\ v_1\in B_1,\ v_2\in B_2\ \mbox{and}\ v_3\in B_3&\\
&0, &\quad \mbox{otherwise},&
\end{aligned}
\right.
\end{eqnarray*}
and
\begin{eqnarray*}
J(v_1\in B_1,v_2 \in B_2)=\left\{
\begin{aligned}
&1, &\quad \mbox{if}\ v_1\in B_1,\ \mbox{and}\ v_2\in B_2,\ &\\
&0, &\quad \mbox{otherwise}.&
\end{aligned}
\right.
\end{eqnarray*}
where $B_1,B_2$ and $B_3$ are intervals or sets.
\begin{theorem}\label{thm3.1}
For the normalizing constant $a_n$ given by \eqref{eq2.2}, we have the following results.
\begin{itemize}
\item[(i).]~When $0<v_1<2$ and $v_2>2v_1$, set $\gamma_1=\min\{2,v_2-v_1,2v_1\}-v_1$, then
$$a_n^{\gamma_1}\Big[a_n^{v_1}(F^n(a_{n}x)-\Phi_{v_1}(x))-k_1(x)\Phi_{v_1}(x)\Big]\to \omega_1(x)\Phi_{v_1}(x),$$
where
$$k_1(x)=-p_1C_{v_1}v_1^{\frac{v_1-1}{2}}T_{v_1+1}(\beta_1\sqrt{v_1+1})x^{-2v_1}$$
and
\begin{eqnarray*}
\omega_1(x)
&=&-\Big(J(0< v_1 <1,2v_1<v_2\le 3v_1)+J(1\le v_1<2,2v_1<v_2\le v_1+2)\Big)A_3x^{-v_2}\\
&&-J(0< v_1 \le 1,v_2\ge 3v_1) \left(\frac{4}{3}p_1^2C_{v_1}^{2}v_{1}^{v_1-1}T_{v_1+1}^{2}(\beta_1\sqrt{v_1+1})x^{-3v_1}
-\frac{k_1^2(x)}{2}\right)\\
&&-J(1\le v_1<2,v_2\ge v_1+2)A_1x^{-v_1-2}.
\end{eqnarray*}
\item[(ii).]~When $v_1>2$ and $v_2>v_1+2$, set $\gamma_2=\min\{4,v_2-v_1,v_1\}-2,$ we have
$$a_n^{\gamma_2}\Big[a_{n}^{2}(F^{n}(a_nx)-\Phi_{v_1}(x))-k_2(x)\Phi_{v_1}(x)\Big]\to \omega_2(x)\Phi_{v_1}(x),$$
where
$$k_2(x)=-A_1x^{-v_1-2}$$
and
\begin{eqnarray*}
\omega_2(x)
&=&-J(v_1\ge 4,v_2\ge v_1+4)\left(A_2x^{-v_1-4}-\frac{k_2^2(x)}{2}\right)-(J(2<v_1<4,v_1+2<v_2\le 2v_1)\\
&&+J(v_1\ge4,v_1+2<v_2\le v_1+4))A_3x^{-v_2}-J(2<v_1\le 4,v_2\ge 2v_1)\\
&&\cdot p_1C_{v_1}v_1^{\frac{v_1-1}{2}}T_{v_1+1}(\beta_1\sqrt{v_1+1})x^{-2v_1}.
\end{eqnarray*}
\item[(iii).]~When $0<v_1<v_2<\min\{2v_1,v_1+2\}$, set $\gamma_3=\min\{2,v_1,2(v_2-v_1),v_3-v_1\}-v_2+v_1,$ then
\begin{eqnarray*}
a_n^{\gamma_3}\Big[a_n^{v_2-v_1}(F^{n}(a_nx)-\Phi_{v_1}(x))-k_3(x)\Phi_{v_1}(x)\Big ]
\to \omega_3(x)\Phi_{v_1}(x),
\end{eqnarray*}
where
$$k_3(x)=-A_3x^{-v_2}$$
and{\small
\begin{eqnarray*}
\omega_3(x)
&=&-I(v_1\ge 2,v_1+1\le v_2<v_1+2,v_3\ge v_1+2)A_1x^{-v_1-2}\\
&&-I(0<v_1\le 2,\frac{3}{2}v_1\le v_2<2v_1,v_3\ge 2v_1)p_1C_{v_1}v_1^{\frac{v_1-1}{2}}T_{v_1+1}(\beta_1\sqrt{v_1+1})x^{-2v_1}\\
&&-\left(I(0<v_1< 2,\frac{v_1+v_3}{2}\leq v_2< 2v_1,v_2<v_3\leq 2v_1) \right. \\
&&\left.  +I(v_1\ge 2,\frac{v_1+v_3}{2}\leq v_2< v_1+2,v_2<v_3\le 2+v_1)\right)A_5x^{-v_3}\\
&&+\Big( I(v_1\geq 2,v_1<v_2\leq 1+v_1,v_3\geq 2+v_1) + I(v_1>2,v_1<v_2\leq \frac{v_1+v_3}{2},v_2<v_3\leq 2+v_1)\\
&& + I(0<v_1<2,v_1<v_2\leq \frac{3v_1}{2},v_3\geq 2v_1)+ I(0<v_1<2,v_1<v_2\leq \frac{v_1+v_3}{2},v_3<2v_1)
\Big) \frac{A_{3}^2 x^{-2v_2}}{2}.
\end{eqnarray*}}
\item[(iv).]~When $v_1=2$ and $v_2>4$, set $\gamma_4=\min\{v_2-2,4\}-2$, we have
$$a_n^{\gamma_4}\Big[a_n^2(F^n(a_{n}x)-\Phi_2(x))-k_4(x)\Phi_2(x)\Big ]\to \omega_4(x)\Phi_2(x),$$
where
$$k_4(x)=-A_1x^{-4}-\frac{1}{2}p_1T_3(\sqrt{3}\beta_1)x^{-4}$$
and
\begin{eqnarray*}
\omega_4(x)&=&-J(v_1=2,4<v_2\le 6)A_3x^{-v_2}-J(v_1=2,v_2\ge 6)\Big(A_2x^{-v_1-4}\\
&&+p_1A_1T_3(\sqrt{3}\beta_1)x^{-6}+\frac{1}{3}p_1^2T_3^2(\sqrt{3}\beta_1)x^{-6}-\frac{k_4^2(x)}{2}
\Big).
\end{eqnarray*}
(v),For $v_1>2$ and $v_2=v_1+2$, set $\gamma_5=\min\{v_1,4,v_3-v_1\}-2$, then
\begin{eqnarray*}
a_n^{\gamma_5}\Big[a_n^2(F^n(a_nx)-\Phi_{v_1}(x))-k_5(x)\Phi_{v_1}(x)\Big]
\to \omega_5(x)\Phi_{v_1}(x),
\end{eqnarray*}
where
$$k_5(x)=-A_1x^{-v_1-2}-A_3x^{-v_2}$$
and
\begin{eqnarray*}
\omega_5(x)
&=&-I(2<v_1\le 4,v_2=v_1+2,v_3\ge 2v_1)p_1C_{v_1}v_1^{\frac{v_1-1}{2}}T_{v_1+1}(\beta_1\sqrt{v_1+1})x^{-2v_1}\\
&&-I(v_1\ge 4,v_2=v_1+2,v_3\ge v_1+4)\left(A_2x^{-v_1-4}+A_4x^{-v_1-4}-\frac{k_5^2(x)}{2}\right)\\
&&-\Big(I(2<v_1<4,v_2=v_1+2,v_1+2<v_3\le 2v_1)+I(v_1\ge 4,v_2=v_1+2,\\
&& \quad v_1+2<v_3
\le v_1+4)\Big)A_5x^{-v_3}.
\end{eqnarray*}
\item[(vi).]~When $0<v_1<2$ and $v_2=2v_1$, set $\gamma_6=\min\{2,2v_1,v_3-v_1\}-v_1$,we have
\begin{eqnarray*}
a_n^{\gamma_6}\Big[a_n^{v_1}(F^n(a_nx)-\Phi_{v_1}(x))-k_6(x)\Phi_{v_1}(x)\Big] \to \omega_6(x)\Phi_{v_1}(x),
\end{eqnarray*}
where
 $$k_6(x)=-A_3x^{-2v_1}-p_1C_{v_1}v_1^{\frac{v_1-1}{2}}T_{v_1+1}(\beta_1\sqrt{v_1+1})x^{-2v_1}$$
 and{\small
\begin{eqnarray*}
\omega_6(x)
&=&-I(1\le v_1<2,v_2=2v_1,v_3\ge v_1+2)A_1x^{-v_1-2}-I(0<v_1\le 1,v_2=2v_1,v_3\ge 3v_1)\\
&&\times \left(2A_3p_1C_{v_1}v_1^{\frac{v_1-1}{2}}T_{v_1+1}(\beta_1\sqrt{v_1+1})x^{-v_2-v_1}
+\frac{4}{3}p_1^2C_{v_1}^2v_1^{v_1-1}T_{v_1+1}^2(\beta_1\sqrt{v_1+1})x^{-3v_1}
-\frac{k_{6}^2(x)}{2}\right)\\
&&-\Big(I(0<v_1<1,v_2=2v_1,2v_1<v_3\le 3v_1)+I(1\le v_1<2,v_2=2v_1,2v_1<v_3\le v_1+2)\Big)
A_5x^{-v_3}.
\end{eqnarray*}}
\item[(vii).]~When $v_1=2$ and $v_2=4$, set $\gamma_7=\min\{v_3-2,4\}-2$,then
$$a_n^{\gamma_7}[a_n^{2}(F^n(a_nx)-\Phi_{2}(x))-k_7(x)\Phi_{2}(x)] \to \omega_7(x)\Phi_{2}(x),$$
where
$$k_7(x)=-A_1x^{-4}-A_3x^{-4}-\frac{1}{2}p_1T_{3}(\sqrt{3}\beta_1)x^{-4}$$
and{\small
\begin{eqnarray*}
\omega_7(x)
&=&-I(v_1=2,v_2=4,4<v_3\le 6)A_5x^{-v_3}-I(v_1=2,v_2=4,v_3\ge6)
\\
&&\times\left(A_2x^{-6}+A_3p_1T_3(\sqrt{3}\beta_1)x^{-6}
+p_1A_1T_3(\sqrt{3}\beta_1)x^{-6}
+\frac{1}{3}p_1^2T_3^2(\sqrt{3}\beta_1)x^{-6}+A_4x^{-6}-\frac{k_{7}^2(x)}{2} \right).
\end{eqnarray*}}
In above (i)-(vii), $A_1$-$A_5$ are those given by Proposition \ref{prop2.2}.
\end{itemize}
\end{theorem}

\begin{proof}[\bf Proof of Theorem 3.1]
Define $h(a_{n};x)=n\log F(a_{n} x)+x^{-v_1}$ with normalized constant $a_{n}$
satisfied \eqref{eq2.2}. From Lemma \ref{Lemma2.2}, it follows that
\begin{eqnarray*}
1-F(a_n)=n^{-1}G(n),
\end{eqnarray*}
where $$G(n)=1+A_1a_n^{-2}+A_2a_n^{-4}+A_3a_n^{-(v_2-v_1)}+A_4a_n^{-(v_2-v_1)-2}+
A_5a_n^{-(v_3-v_1)}+O(a_n^{-\eta})$$ and $\eta=\min\{6,v_4-v_1,v_2-v_1+4,v_3-v_1+2\}$.
Let
\begin{eqnarray}\label{eq3.1}\nonumber
B_n(x)&=&1+A_1a_n^{-2}x^{-2}+A_2a_n^{-4}x^{-4}+A_3a_n^{-(v_2-v_1)}x^{-(v_2-v_1)}+A_4a_n^{-(v_2-v_1)-2}x^{-(v_2-v_1)-2}\\
&&+A_5a_n^{-(v_3-v_1)}x^{-(v_3-v_1)}+O(a_n^{-\eta}),
\end{eqnarray}
then we can get
\begin{eqnarray}\label{eq3.2}
n(1-F(a_nx))=B_n(x)x^{-v_1}.
\end{eqnarray}
By using \eqref{eq3.1} and Lemma \ref{Lemma2.2}, we have
\begin{eqnarray}\label{eq3.3}\nonumber
&&B_n(x)(1-F(a_nx))\\\nonumber
&=&2p_1C_{v_1}v_1^{\frac{v_1-1}{2}}T_{v_1+1}(\beta_1\sqrt{v_1+1})a_n^{-v_1}x^{-v_1}\\\nonumber
&&\cdot[1+2A_1x^{-2}a_n^{-2}+2A_3x^{-(v_2-v_1)}a_n^{-(v_2-v_1)}+(2A_2+A_1^2)x^{-4}a_n^{-4}\\
&&+2A_5x^{-(v_3-v_1)}a_n^{-(v_3-v_1)}+A_3^2x^{-2(v_2-v_1)}a_n^{-2(v_2-v_1)}\\\nonumber
&&+2(A_4+A_1A_3)x^{-(v_2-v_1)-2}a_n^{-(v_2-v_1)-2} +o(a_{n}^{-\eta_1})]
\end{eqnarray}
and
\begin{eqnarray}\label{eq3.4}\nonumber
&&B_n(x)(1-F(a_nx))^2\\\nonumber
&=&4p_1^2C_{v_1}^2v_1^{v_1-1}T_{v_1+1}^2(\beta_1\sqrt{v_1+1})x^{-2v_1}a_n^{-2v_1}
(1+3A_1x^{-2}a_n^{-2}\\
&&+3A_3x^{-(v_2-v_1)}a_n^{-(v_2-v_1)}+o(a_n^{-\eta_2}))
\end{eqnarray}
where $\eta_1=\min\{4,v_2-v_1+2,2(v_2-v_1),v_3-v_1\}$ and $\eta_2=\min\{2,v_2-v_1\}$.

Combining \eqref{eq3.2},\eqref{eq3.3} and \eqref{eq3.4}, we have
\begin{eqnarray}\label{eq3.5}
h(a_n;x)
&=&n\log F(a_nx)+x^{-v_1}\nonumber\\
&=&\left[(1-B_n(x))-\frac{1}{2}B_n(x)(1-F(a_nx))-\frac{1}{3}B_n(x)(1-F(a_nx))^2(1+o(1))\right]x^{-v_1}\nonumber\\
&=&-A_1x^{-v_1-2}a_n^{-2}-A_2x^{-v_1-4}a_n^{-4}-A_3x^{-v_2}a_n^{-(v_2-v_1)}-A_4x^{-v_2-2}a_n^{-(v_2-v_1)-2}\nonumber\\
&&-A_5x^{-v_3}a_n^{-(v_3-v_1)}-p_1C_{v_1}v_1^{\frac{v_1-1}{2}}T_{v_1+1}(\beta_1\sqrt{v_1+1})x^{-2v_1}a_n^{-v_1}\nonumber\\
&&-2A_1p_1C_{v_1}v_1^{\frac{v_1-1}{2}}T_{v_1+1}(\beta_1\sqrt{v_1+1})x^{-2v_1-2}a_n^{-v_1-2}\nonumber\\
&&-2A_3p_1C_{v_1}v_1^{\frac{v_1-1}{2}}T_{v_1+1}(\beta_1\sqrt{v_1+1})x^{-v_2-v_1}a_n^{-v_2}\nonumber\\
&&-\frac{4}{3}p_1^2C_{v_1}^2v_1^{v_1-1}T_{v_1+1}^2(\beta_1\sqrt{v_1+1})x^{-3v_1}a_n^{-2v_1}
+o(a_n^{-\eta_3}),
\end{eqnarray}
where $\eta_3=\max\{4, v_3-v_1, v_2-v_1+2, v_2,2v_1, v_1+2\}.$

Now we only prove the theorem for the case of $0<v_1<v_2<\min\{v_1+2,2v_1\}$, and the
proofs of the rest cases are similar.
Note that \eqref{eq3.5} implies $\lim_{n\to \infty} h(a_n;x)=0$. Then we get
$$\left|\sum_{i=3}^{\infty}\frac{h^{i-3}(a_n;x)}{i!}\right|<\exp{(h(a_n;x))}\to 1$$
as $n\to \infty.$ Hence
\begin{eqnarray*}
& & F^n(a_{n}x)-\Phi_{v_1}(x)\\
&=& \left( h(a_{n};x)+h^2(a_{n};x)\left( \frac{1}{2}+h(a_{n};x)\sum_{i=3}^{\infty} \frac{h^{i-3}(a_{n};x)}{i!}  \right)  \right)\Phi_{v_1}(x) \\
&=& \Big[ -A_{3}x^{-v_{2}}a_{n}^{-(v_2-v_1)} + \Big( -I(v_1\ge 2,v_1+1\le v_2<v_1+2,v_3\ge v_1+2)A_1x^{-v_1-2}\\
&& -I(0<v_1\le 2,\frac{3}{2}v_1\le v_2<2v_1,v_3\ge 2v_1)p_1C_{v_1}v_1^{\frac{v_1-1}{2}}T_{v_1+1}(\beta_1\sqrt{v_1+1})x^{-2v_1}\\
&& -\left(I(0<v_1< 2,\frac{v_1+v_3}{2}\leq v_2< 2v_1,v_2<v_3\leq 2v_1) \right. \\
&&\left.  +I(v_1\ge 2,\frac{v_1+v_3}{2}\leq v_2< v_1+2,v_2<v_3\le 2+v_1)\right)A_5x^{-v_3}\\
&& +\Big( I(v_1\geq 2,v_1<v_2\leq 1+v_1,v_3\geq 2+v_1) + I(v_1>2,v_1<v_2\leq \frac{v_1+v_3}{2},v_2<v_3\leq 2+v_1)\\
&&  + I(0<v_1<2,v_1<v_2\leq \frac{3v_1}{2},v_3\geq 2v_1)+ I(0<v_1<2,v_1<v_2\leq \frac{v_1+v_3}{2},v_3<2v_1)
\Big) \frac{A_{3}^2 x^{-2v_2}}{2}  \Big)\\
&&\times a_{n}^{-\min(2,v_1,2(v_2-v_1),v_3-v_1)}(1+o(1))  \Big]\Phi_{v_1}(x)
\end{eqnarray*}
for large $n$, which implies that
\[ a_{n}^{\gamma_3}\left[a_{n}^{v_2-v_1}\left(F^n(a_{n}x)-\Phi_{v_1}(x)\right)
-k_3(x)\Phi_{v_1}(x)\right]\to \omega_3(x)\Phi_{v_1}(x) \]
as $n\to \infty$. Here $\gamma_3=\min\{2,v_1,2(v_2-v_1),v_3-v_1\}-v_2+v_1$, $k_3(x)=-A_3x^{-v_2}$ and{\small
\begin{eqnarray*}
\omega_3(x)
&=& -I(v_1\ge 2,v_1+1\le v_2<v_1+2,v_3\ge v_1+2)A_1x^{-v_1-2}\\
&&-I(0<v_1\le 2,\frac{3}{2}v_1\le v_2<2v_1,v_3\ge 2v_1)p_1C_{v_1}v_1^{\frac{v_1-1}{2}}T_{v_1+1}(\beta_1\sqrt{v_1+1})x^{-2v_1}\\
&&-\left(I(0<v_1< 2,\frac{v_1+v_3}{2}\leq v_2< 2v_1,v_2<v_3\leq 2v_1) \right. \\
&&\left. \ \  +I(v_1\ge 2,\frac{v_1+v_3}{2}\leq v_2< v_1+2,v_2<v_3\le 2+v_1)\right)A_5x^{-v_3}\\
&&+\Big( I(v_1\geq 2,v_1<v_2\leq 1+v_1,v_3\geq 2+v_1) + I(v_1>2,v_1<v_2\leq \frac{v_1+v_3}{2},v_2<v_3\leq 2+v_1)\\
&& \ \ + I(0<v_1<2,v_1<v_2\leq \frac{3v_1}{2},v_3\geq 2v_1)+ I(0<v_1<2,v_1<v_2\leq \frac{v_1+v_3}{2},v_3<2v_1)
\Big) \frac{A_{3}^2 x^{-2v_2}}{2}.
\end{eqnarray*}}
The proof is complete.
\end{proof}

Based on Theorem \ref{thm3.1}, the asymptotic expansion for the pdf of $M_n$ can also be derived. Let
\begin{eqnarray}\label{eq3.6}
g_n(x)=na_nF^{n-1}(a_nx)f(a_nx)
\end{eqnarray}
denote the pdf of the normalized maximum, and define
\begin{eqnarray}\label{eq3.7}
\Delta_n(g_n,\Phi_{v}^{'};x)=g_n(x)-\Phi_{v}^{'}(x)
\end{eqnarray}
with $\Phi_{v}^{'}(x)=vx^{-v-1}\Phi_{v}(x)$. From Proposition2.5 of Resnick(1987),
it follows that
$\Delta_n(g_n,\Phi_{v}^{'};x) \to 0$ as $n\to \infty$. The higher-order asymptotic expansion of the pdf
of $M_n$ is given as follows.

\begin{theorem}\label{thm3.2}
For the normalized constant $a_{n}$ given by \eqref{eq2.2}, we have the following results:

\noindent (i), if $0<v_1<2$ and $v_2>2v_1$, set $\gamma_1=\min\{2,v_2-v_1,2v_1\}-v_1$, then
$$a_n^{\gamma_1}[a_n^{v_1}\Delta_n(g_n,\Phi'_{v_1};x)-s_1(x)\Phi_{v_1}^{'}(x)]\to q_1(x)\Phi_{v_1}^{'}(x),$$
where$$s_1(x)=p_1C_{v_1}v_1^{\frac{v_1-1}{2}}T_{v_1+1}(\beta_1\sqrt{v_1+1})(2-x^{-v_1})x^{-v_1}$$
and
\begin{eqnarray*}
q_1(x)
&=&\Big[J(0<v_1<1,2v_1<v_2\le 3v_1)\\
&& +J(1\le v_1< 2,2v_1<v_2\le v_1+2)\Big]\left(\frac{v_2}{v_1}x^{-(v_2-v_1)}-x^{-v_2}\right)A_3\\
&&+J(1\le v_1< 2,v_2\ge v_1+2)(K_1x^{-2}-A_1x^{-v_1-2})\\
&&+J(0< v_1\le 1,v_2\ge 3v_1)p_1^2C_{v_1}^2v_1^{v_1-1}T_{v_1+1}^2(\beta_1\sqrt{v_1+1})\left(4-\frac{10}{3}x^{-v_1}+\frac{1}{2}x^{-2v_1}\right)x^{-2v_1}.
\end{eqnarray*}
(ii), if $v_1>2$ and $v_2>v_1+2$, set $\gamma_2=\min\{4,v_2-v_1,v_1\}-2$, we
$$a_n^{\gamma_2}[a_n^{2}\Delta_n(g_n,\Phi_{v_1}^{'};x)-s_2(x)\Phi_{v_1}^{'}(x)]\to q_2(x)\Phi_{v_1}^{'}(x),$$
where
$$s_2(x)=K_1x^{-2}-A_1x^{-v_1-2}$$
and
\begin{eqnarray*}
q_2(x)
&=&J(v_1\ge4,v_2\ge v_1+4)\left[K_3x^{-4}-(A_2+K_1A_1)x^{-v_1-4}+\frac{1}{2}A_1^2x^{-2v_1-4}\right]\\
&&+\Big(J(v_1\ge4,v_1+2<v_2\le v_1+4)+J(2<v_1<4,v_1+2<v_2\le 2v_1)\Big)
A_3\left(\frac{v_2}{v_1}x^{-(v_2-v_1)}-x^{-v_2}\right)\\
&&+J(2<v_1\le4,v_2\ge 2v_1)p_1C_{v_1}v_1^{\frac{v_1-1}{2}}T_{v_1+1}(\beta_1\sqrt{v_1+1})(2-x^{-v_1})x^{-v_1}.
\end{eqnarray*}
(iii), if $0<v_1<v_2<\min\{v_1+2,2v_1\}$, set $\gamma_3=\min\{2,v_1,2(v_2-v_1),v_3-v_1\}-v_2+v_1$, then
$$a_n^{\gamma_3}[a_n^{v_2-v_1}\Delta_n(g_n,\Phi_{v_1}^{'};x)-s_3(x)\Phi_{v_1}^{'}(x)]\to q_3(x)\Phi_{v_1}^{'}(x)$$
where
\begin{eqnarray}\label{addeq1}
s_3(x)=\left(\frac{v_2}{v_1}x^{-(v_2-v_1)}-x^{-v_2}\right)A_3
\end{eqnarray}
and
\begin{eqnarray}\label{addeq2}
q_3(x)
&=&\Big(I(0<v_1<2,\frac{v_1+v_3}{2}\leq v_2 < 2v_1,v_3\le 2v_1)\nonumber\\
&& +I(v_1\ge2, \frac{3v_1}{2}\leq v_2< 2v_1,v_3\ge 2v_1)\Big)
\left(\frac{v_{3}}{v_{1}}x^{-(v_3-v_1)}-x^{-v_3}\right)A_5\nonumber\\
&&+I\left(0<v_1\le 2,\frac{3}{2}v_1\le v_2< 2v_1,v_3\ge 2v_1\right) p_1C_{v_1}v_1^{\frac{v_1-1}{2}}T_{v_1+1}(\beta_1\sqrt{v_1+1})(2-x^{-v_1})x^{-v_1}\nonumber\\
&&+I(2\leq v_1,v_1+1\le v_2< v_1+2,v_3\ge v_1+2)(K_1x^{-2}-A_1x^{-v_1-2})\nonumber\\
&&+\Big[I\left(0<v_1<2,v_1<v_2\le \frac{3}{2}v_1,v_3\ge 2v_1\right)\nonumber \\
&&
+I\left(0<v_1<2,v_1< v_2\le \frac{v_1+v_3}{2}, v_2<v_3< 2v_1\right)\nonumber\\
&&+I(v_1\ge2,v_1<v_2\le v_1+1, v_3\ge 2+v_1)\nonumber\\
&& +I(v_1> 2, v_1< v_2 \le \frac{v_1+v_3}{2},v_2<v_3\le v_1+2)\Big]\left( \frac{x^{-2v_2}}{2} -\frac{v_2}{v_1}x^{-2v_2+v_1} \right)A_3^2.
\end{eqnarray}
(iv), if $v_1=2$, and $v_2>4$, set $\gamma_4=\min\{v_2-2,4\}-2$, we have
$$a_n^{\gamma_4}[a_n^{2}\Delta_n(g_n,\Phi_{v_1}^{'};x)-s_4(x)\Phi_{v_1}^{'}(x)]\to q_4(x)\Phi_{v_1}^{'}(x)$$
where
$$s_4(x)=-A_1x^{-4}-\frac{1}{2}p_1T_3(\sqrt{3}\beta_1)x^{-4}+K_1x^{-2}+p_1T_3(\sqrt{3}\beta_1),$$
and
\begin{eqnarray*}
q_4(x)&=&J(v_1=2,4<v_2\le 6)A_3\left(\frac{v_2}{2}x^{-(v_2-2)}-x^{-v_2}\right)+J(v_1=2,v_2\ge 6)\\
&&\left[-\left(A_2+A_1K_1+\left(2A_1+\frac{K_1}{2}\right)p_1T_3(\sqrt{3\beta_1})+\frac{5}{6}p_1^2T_3(\sqrt{3}\beta_1)\right)x^{-6}
+\left(K_3\right.\right.\\
&&\left.\left.+(A_1+K_1)p_1T_3(\sqrt{3}\beta_1)+p_1^2T_3^2(\sqrt{3}\beta_1)\right)x^{-4}+\frac{\left(2A_1+p_1T_3(\sqrt{3}\beta_1)
\right)^2x^{-8}}{8}\right].
\end{eqnarray*}
(v), if $v_1>2$ and $v_2=v_1+2$, set $\gamma_5=\min\{v_1,4,v_3-v_1\}-2$, then
$$a_n^{\gamma_5}[a_n^{2}\Delta_n(g_n,\Phi_{v_1}^{'};x)-s_5(x)\Phi_{v_1}^{'}(x)]\to q_5(x)\Phi_{v_1}^{'}(x)$$
where
$$s_5(x)=\left(K_1+A_3\left(1+\frac{2}{v_1}\right)\right)x^{-2}-(A_1+A_3)x^{-v_1-2}$$
and
\begin{eqnarray*}
q_5(x)
&=&I(2<v_1<4,v_2=v_1+2,v_3\ge 2v_1)p_1C_{v_1}v_1^{\frac{v_1-1}{2}}T_{v_1+1}(\beta_1\sqrt{v_1+1})(2x^{-v_1}-x^{-2v_1})\\
&&+\Big(I(2<v_1<4,v_2=v_1+2,v_1+2<v_3\le 2v_1) \\
&& +I(v_1\ge 4,v_2=v_1+2,v_1+2<v_3 \le v_1+4)\Big)\left(\frac{v_3}{v_1}x^{-(v_3-v_1)}-x^{-v_3}\right)A_5\\
&& +I(v_1\ge 4,v_2=v_1+2,v_3\ge v_1+4)\Big[\frac{1}{2}(A_1+A_3)^2x^{-2v_1-4}+(K_2+K_3)x^{-4}\\
&& -\left(\left(\left(1+\frac{2}{v_1}\right)A_3+K_1\right)(A_1+A_3)+A_2+A_4\right)x^{-v_1-4}\Big].
\end{eqnarray*}
(vi), if $0<v_1<2$ and $v_2=2v_1$, set $\gamma_6=\min\{v_1,4,v_3-v_1\}-v_1$, then
$$a_n^{\gamma_6}[a_n^{v_1}\Delta_n(g_n,\Phi_{v_1}^{'};x)-s_6(x)\Phi_{v_1}^{'}(x)]\to q_6(x)\Phi_{v_1}^{'}(x)$$
where
$$s_6(x)=p_1C_{v_1}v_1^{\frac{v_1-1}{2}}T_{v_1+1}(\beta_1\sqrt{v_1+1})(2x^{-v_1}-x^{-2v_1})+A_3(2x^{-v_1}-x^{-2v_1})$$
and
\begin{eqnarray*}
q(x)&=&I(1\le v_1<2,v_2=2v_1,v_3 \ge v_1+2)(K_1x^{-2}-A_1x^{-v_1-2})\\
&&+(I(0< v_1<1,v_2=2v_1,2v_1<v_3 \le 3v_1)+I(1\le v_1<2,v_2=2v_1,2v_1<v_3 \\
&&\le v_1+2))A_5\left(\frac{v_3}{v_1}x^{-(v_3-v_1)}-x^{-v_3}\right)+I(0<v_1\le 1,v_2=2v_1,v_3\ge 3v_1)\\
&&\left[p_1^2C_{v_1}^2v_1^{v_1-1}T_{v_1+1}^2(\beta_1\sqrt{v_1+1})x^{-2v_1}\left(4-\frac{10}{3}x^{-v_1}\right)+p_1C_{v_1}v_1^{\frac{v_1-1}{2}}T_{v_1+1}(\beta_1\sqrt{v_1+1})
\right.\\
&&\left.6A_3(-x^{-3v_1}+x^{-2v_1})-2A_3^2x^{-3v_1}+\frac{1}{2}(A_3+p_1C_{v_1}v_1^{\frac{v_1-1}{2}}T_{v_1+1}(\beta_1\sqrt{v_1+1}))^2x^{-4v_1}\right].
\end{eqnarray*}
(vii), if $v_1=2$ and $v_2=4$, set $\gamma_7=\min\{v_3-v_1,4\}-2$, then
$$a_n^{\gamma_7}[a_n^{2}\Delta_n(g_n,\Phi_{v_1}^{'};x)-s_7(x)\Phi_{v_1}^{'}(x)]\to q_7(x)\Phi_{v_1}^{'}(x)$$
where
$$s_7(x)=(K_1+2A_3+p_1T_3(\sqrt{3}\beta_1))x^{-2}-\left(A_1+A_3+\frac{1}{2}p_1T_3(\sqrt{3}\beta_1)\right)x^{-4}$$
and
\begin{eqnarray*}
q_7(x)
&=&I(v_1=2,v_2=4,4<v_3\le 6)A_5\left(\frac{v_3}{2}x^{-(v_3-2)}-x^{-v_3}\right)\\
&&+I(v_1=2,v_2=4,v_3\ge 6)\Big[\frac{1}{2}\left((A_1+A_3)+\frac{1}{2}p_1T_3(\sqrt{3}\beta_1)\right)^2x^{-8}
-\Big(A_4+A_2+2A_3^2\\
&& +2A_1A_3+A_1K_1+A_3K_1+\left(2A_1+3A_3+\frac{K_1}{2}\right)p_1T_3(\sqrt{3}\beta_1)
+ \frac{5}{6}p_1^2T_3^2(\sqrt{3}\beta_1)\Big)x^{-6}\\
&& +(K_2+K_3+(A_1+K_1+3A_3)p_1T_3(\sqrt{3}\beta_1)+p_1^2T_3^2(\sqrt{3}\beta_1))x^{-4}\Big].
\end{eqnarray*}
In (i)-(vii), $A_1$-$A_5$ are those defined by Lemma \ref{lemma2.2} and  and $K_1, K_2, K_3$ are given by Lemma \ref{lemma3.1} below.
\end{theorem}

To prove Theorem 3.2, we need the following lemma.
\begin{lemma}\label{lemma3.1}
Let $F(x)$ denote the cdf of random variable T defined by \eqref{eq1.3} and $f(x)$ the pdf of $T$. With normalized constant $a_{n}$ given by \eqref{eq2.2}, we have
\begin{eqnarray*}
na_nf(a_nx)&=&v_1x^{-v_1-1}\left[1+K_1x^{-2}a_n^{-2}+\frac{v_2}{v_1}A_3x^{-(v_2-v_1)}a_n^{-(v_2-v_1)}+K_2x^{-(v_2-v_1)-2}a_n^{-(v_2-v_1)-2}\right.\\
&&\left.+\frac{v_3}{v_1}A_5x^{-(v_3-v_1)}a_n^{-(v_3-v_1)}
+K_3x^{-4}a_n^{-4}+O(a_n^{-\eta})\right]
\end{eqnarray*}
for large $x$, where $\eta=\min\{6,v_2-v_1+4,v_3-v_1+2,v_4-v_1\}$, and
\begin{eqnarray*}
K_1&=&-\frac{v_1\left(T_{v_1+1}(\beta_1\sqrt{v_1+1})(v_1+1)
+C_{v_1+1}(1+\beta_1^2)^{-\frac{v_1+2}{2}}\beta_1\sqrt{v_1+1}\right)}
{2T_{v_1+1}(\beta_1\sqrt{v_1+1})},\\
K_2&=&-\frac{p_2C_{v_2}v_2^{\frac{v_2+3}{2}}\left(T_{v_2+1}(\beta_2\sqrt{v_2+1})(v_2+1)+C_{v_2+1}(\beta_2\sqrt{v_2+1})(1+\beta_2^2)^{-\frac{v_2+2}{2}}\right)}
{2p_1C_{v_1}v_1^{\frac{v_1+1}{2}}T_{v_1+1}(\beta_1\sqrt{v_1+1})},\\
K_3&=&\frac{\frac{v^2(v_1+3)(v_1+1)}{8}T_{v_1+1}(\beta_1\sqrt{v_1+1})+C_{v_1+1}(\beta_1\sqrt{v_1+1})(1+\beta_1^2)^{-\frac{v_1+2}{2}}
\left(\frac{2v_1^3+v_1^3\beta_1^2+3v_1^2\beta_1^2+5v_1^2}{8(1+\beta_1^2)}\right)}{T_{v_1+1}(\beta_1\sqrt{v_1+1})}.
\end{eqnarray*}
\end{lemma}
\begin{proof}
Using Taylor expansion with Lagrange remainder term, for large $x$ we have
\begin{eqnarray*}
&&T_{v+1}\left( \frac{\beta\sqrt{v+1}}{\sqrt{x^2+v}}x  \right)\\
&=& T_{v+1}(\beta \sqrt{v+1})+C_{v+1}\beta \sqrt{v+1}(1+\beta^2)^{-\frac{v+2}{2}}
\left( -\frac{v}{2}x^{-2}+\frac{3v^2+v^2\beta^2-v^3\beta^2}{8(1+\beta^2)}x^{-4}
+O(x^{-6}) \right)
\end{eqnarray*}
and
\[\left( 1+\frac{v}{x^2} \right)^{-\frac{v+1}{2}}
=1-\frac{v(v+1)}{2}x^{-2}+\frac{(v+1)(v+3)v^2}{8}x^{-4}+O(x^{-6}),\]
which implies that
\begin{eqnarray}\label{eq3.8}
&&\left( 1+\frac{v}{x^2} \right)^{-\frac{v+1}{2}}
T_{v+1}\left( \frac{\beta\sqrt{v+1}}{\sqrt{x^2+v}}x  \right)x^{-v-1} \nonumber\\
&=& T_{v+1}(\beta \sqrt{v+1}) x^{-v-1} \left[ 1- \left( \frac{v(v+1)}{2} +\frac{vC_{v+1}\beta\sqrt{v+1}(1+\beta^2)^{-\frac{v+2}{2}}}{2T_{v+1}(\beta\sqrt{v+1})}  \right)x^{-2} \right. \nonumber\\
& & \left. + \left( \frac{(v+1)(v+3)v^2}{8} +\frac{C_{v+1}\beta \sqrt{v+1}(1+\beta^2)^{-\frac{v+4}{2}}(3v^2+v^2\beta^2-v^3\beta^2)}{8T_{v+1}(\beta\sqrt{v+1})} \right. \right. \nonumber\\
& & \quad \left. \left. +\frac{v^2(v+1)^{\frac{3}{2}}\beta C_{v+1}(1+\beta^2)^{-\frac{v+2}{2}}}{4T_{v+1}(\beta\sqrt{v+1})}\right)x^{-4} + O(x^{-6}) \right] \nonumber\\
&=& T_{v+1}(\beta \sqrt{v+1}) x^{-v-1} \left[ 1+C_{1}(v,\beta)x^{-2}+C_{2}(v,\beta)x^{-4} + O(x^{-6})  \right].
\end{eqnarray}
Here,
\[C_{1}(v,\beta)=-\frac{v(v+1)}{2}-\frac{vC_{v+1}\beta \sqrt{v+1}(1+\beta^2)^{-\frac{v+2}{2}}}{2T_{v+1}(\beta\sqrt{v+1})}\]
and
\begin{eqnarray*}
C_{2}(v,\beta) &=& \frac{(v+1)(v+3)v^2}{8} +\frac{C_{v+1}\beta \sqrt{v+1}(1+\beta^2)^{-\frac{v+4}{2}}(3v^2+v^2\beta^2-v^3\beta^2)}{8T_{v+1}(\beta\sqrt{v+1})} \nonumber\\
& &  +\frac{v^2(v+1)^{\frac{3}{2}}\beta C_{v+1}(1+\beta^2)^{-\frac{v+2}{2}}}{4T_{v+1}(\beta\sqrt{v+1})}.
\end{eqnarray*}
Combining \eqref{eq1.1}, \eqref{eq1.2}, \eqref{eq1.4} and \eqref{eq3.8}, we can get
\begin{eqnarray*}
f(a_{n}x)
&=&\sum_{i=1}^{r}2p_{i}C_{v_{i}}v_{i}^{\frac{v_{i}+1}{2}}\left(  1+\frac{v_{i}}{(a_{n}x)^2} \right)^{-\frac{v_{i}+1}{2}}T_{v_{i}+1}\left( \frac{\beta_{i}a_{n}x\sqrt{v_{i}+1}}{\sqrt{v_{i}+(a_{n}x)^2}} \right)(a_{n}x)^{-v_{i}-1}\\
&=& 2p_{1}C_{v_{1}}v_{1}^{\frac{v_{1}+1}{2}}T_{v_{1}+1}(\beta_1\sqrt{v_1+1})(a_{n}x)^{-v_{1}-1}
\left[  1+C_1(v_1,\beta_1)(a_n x)^{-2} +C_2(v_1,\beta_1)(a_{n}x)^{-4}\right. \\
& & \quad \left. +\sum_{i=2}^{r}\frac{p_{i}C_{v_i}v_{i}^{\frac{v_i+1}{2}}T_{v_i+1}(\beta_i\sqrt{v_i+1})}
{p_{1}C_{v_1}v_{1}^{\frac{v_1+1}{2}}T_{v_1+1}(\beta_1\sqrt{v_1+1})}(a_{n}x)^{-(v_i-v_1)}
\left(  1+C_1(v_i,\beta_i)(a_{n}x)^{-2} \right. \right. \\
& & \quad \left. \left. +C_2(v_i,\beta_i)(a_{n}x)^{-4} + O(a_{n}^{-6}) \right)+O(a_{n}^{-6}) \right]\\
&=& 2p_{1}C_{v_1}v_{1}^{\frac{v_1+1}{2}}T_{v_1+1}(\beta_1\sqrt{v_1+1})(a_{n}x)^{-v_1-1}
\left[  1+K_1(a_{n}x)^{-2}+\frac{v_2}{v_1}A_{3}(a_{n}x)^{-(v_2-v_1)} \right. \\
& & \left. +K_2(a_{n}x)^{-(v_2-v_1)-2} +\frac{v_3}{v_1}A_5(a_n x)^{-(v_3-v_1)} +K_3(a_{n}x)^{-4} + O(a_{n}^{-\eta}) \right]
\end{eqnarray*}
for large $n$, where $\eta=\min\{6,v_2-v_1+4,v_3-v_1+2,v_4-v_1\}$, and $K_1$, $K_2$, $K_3$ are defined as above.
With normalized constant $a_n$ given by \eqref{eq2.2}, the desired result can be derived, which complete the proof.
\end{proof}

\begin{proof}[\bf Proof of Theorem3.2]
By Lemma \ref{lemma2.2} and Lemma \ref{lemma3.1}, we have
\begin{eqnarray}\label{eq3.9}
\frac{na_nf(a_nx)}{F(a_nx)}
&=&vx^{-v_1-1}\left[1+K_1x^{-2}a_n^{-2}+2p_1C_{v_1}v_1^{\frac{v_1-1}{2}}T_{v_1+1}(\beta_1\sqrt{v_1+1})x^{-v_1}a_n^{-v_1}\right.\nonumber\\
&&+\frac{v_2}{v_1}A_3x^{-(v_2-v_1)}a_n^{-(v_2-v_1)}+K_2x^{-(v_2-v_1)-2}a_n^{-(v_2-v_1)-2}+K_3x^{-4}a_n^{-4}\nonumber\\
&&+\frac{v_3}{v_1}A_5x^{-(v_3-v_1)}a_n^{-(v_3-v_1)}+2p_1C_{v_1}v_1^{\frac{v_1-1}{2}}T_{v_1+1}(\beta_1\sqrt{v_1+1})(A_1+K_1)x^{-v_1-2}a_n^{-v_1-2}\nonumber\\
&&+2A_3p_1C_{v_1}v_1^{\frac{v_1-1}{2}}T_{v_1+1}(\beta_1\sqrt{v_1+1})\left(\frac{v_2}{v_1}+1\right)x^{-v_2}a_n^{-v_2}\nonumber\\
&&\left.+4p_1^2C_{v_1}^2v_1^{v_1-1}T_{v_1+1}^2(\beta_1\sqrt{v_1+1})x^{-2v_1}a_n^{-2v_1} + O\left((a_{n}^{-\eta_3}\right)\right]
\end{eqnarray}
for large $n$, where $\eta_3=\min(v_2-v_1+2,v_3-v_1,4,v_2,v_1+2,2v_1)$.

Now we consider the case $0<v_1<v_2<\min\{2v_1,v_1+2\}$. Proofs of the rest cases are similar, and we omit here.
From \eqref{eq3.2}, \eqref{eq3.6}, \eqref{eq3.7}, \eqref{eq3.9} and Theorem \ref{thm3.1}, it follows that
\begin{eqnarray*}
& & \Delta_n\left( g_n,\Phi'_{v_1};x \right) \\
&=& na_{n}f(a_{n}x)\frac{F^{n}(a_{n}x)}{F(a_{n}x)}-\Phi'_{v_1}(x)\\
&=& \Phi'_{v_1}(x)\Big[  \frac{v_2}{v_1}A_3x^{-(v_2-v_1)}a_{n}^{-(v_2-v_1)} + K_{1}x^{-2}a_{n}^{-2}
+ 2p_{1}C_{v_1}v_{1}^{\frac{v_1+1}{2}}T_{v_1+1}(\beta_1\sqrt{v_1+1})x^{-v_1}a_{n}^{-v_1} \\
& &  +\frac{v_3}{v_1}A_5x^{-(v_3-v_1)}a_{n}^{-(v_3-v_1)} - A_3x^{-v_2}a_{n}^{-(v_2-v_1)}
-\frac{v_2}{v_1}A_{3}^{2}x^{-2v_2+v_1}a_{n}^{-2(v_2-v_1)} \\
& & +\omega_3(x)a_{n}^{-\min\{2,v_1,v_3-v_1,2(v_2-v_1)\}}(1+o(1)) \Big]  \\
&=& \Phi'_{v_1}(x)\left\{ \left( \frac{v_2}{v_1}x^{-(v_2-v_1)} -x^{-v_2} \right)A_{3}a_{n}^{-(v_2-v_1)}
+ \left[ \left( I\left(v_1\ge 2, \frac{v_1+v_3}{2}\le v_2<v_1+2, v_3\le 2+v_1\right) \right. \right. \right.
\\
& & \left.+I\left(0<v_1<2,\frac{v_1+v_3}{2}\le v_2 <2v_1, v_3\le 2v_1  \right)\right)\left( \frac{v_3}{v_1}x^{-(v_3-v_1)}-x^{-v_3} \right)A_5\\
&&
+ I(v_1\ge 2,v_1+1\le v_2<v_1+2,v_3\geq 2+v_1,v_3\ge 2+v_1)\left( K_1x^{-2}-A_1x^{-v_1-2} \right) \\
& &
+ I\left(0<v_1\le 2, \frac{3v_1}{2}\le v_2<2v_1, v_3\leq 2v_1\right)
p_{1}C_{v_1}v_{1}^{\frac{v_1+1}{2}}T_{v_1+1}(\beta_1\sqrt{v_1+1})x^{-v_1}\left( 2-x^{-v_1} \right) \\
& &
+\left(I(v_1\ge 2,v_1<v_2\le 1+v_1,v_3\ge 2+v_1) + I\left( v_1>2,v_1<v_2\leq \frac{v_1+v_3}{2},v_3\leq 2+v_1\right) \right.\\
& &
\left. + I\left( 0<v_1<2,v_1<v_2\le \frac{3v_1}{2},v_3\ge 2v_1 \right)
+ I\left( 0<v_1<2,v_1<v_2\le \frac{v_1+v_3}{2},v_3<2v_1 \right) \right)\\
& & \left. \left. \times \left( \frac{1}{2}x^{-2v_2}-\frac{v_2}{v_1}x^{-2v_2+v_1} \right)A_3^2 \right]
a_{n}^{-\min\{2,v_1,v_3-v_1,2(v_2-v_1)\}}(1+o(1))   \right\},
\end{eqnarray*}
where $A_1-A_5$, $K_1$, $\gamma_3$ and $\omega_3(x)$ are given by Theorem \ref{thm3.1} and Lemma \ref{lemma3.1}.
Thus, we can get
\[a_{n}^{\gamma_3}\left[ a_{n}^{v_2-v_1}\Delta_n(g_n,\Phi'_{v_1};x)-s_3(x)\Phi'_{v_1}(x) \right] \to q_3(x)\Phi'_{v_1}(x)\]
as $n\to \infty$, where $s_3(x)$ and $q_3(x)$ are those given by \eqref{addeq1} and \eqref{addeq2}, respectively.

The proof is complete.
\end{proof}

Proposition \ref{prop2.2} shows that $F\in D_{p}(\Phi_{1})$. Noting that
$F\left( \alpha_n|x|^{\beta_n}\sign (x) \right)=F\left( a_nx^{\frac{1}{v_1}} \right)$ and
$\Phi_1(x)=\Phi_{v_1}\left( x^{\frac{1}{v_1}} \right)$ for $x>0$, where the normalized constants
$\alpha_n=a_n$ and $\beta_n=\frac{1}{v_1}$. From Theorem \ref{thm3.1}, one
can easily get the higher-order expansions of the cdf of $M_n$ under power
normalization.

\begin{remark}\label{rem3.1}
With the normalized constants $\alpha_n=a_n$ and $\beta_n=\frac{1}{v_1}$, we have
$F^{n}\left( \alpha_n|x|^{\beta_n}\sign(x) \right)-\Phi_{1}(x)=F^n\left( a_n x^{\frac{1}{v_1}} \right)
-\Phi_{v_1}\left( x^{\frac{1}{v_1}} \right)$ for $x>0$, where $a_{n}$ is given by \eqref{eq2.2}.
Hence, the higher-order expansions of the cdf of extremes from the mixed skew-t sample under power
normalization can be derived through replacing $x$
by $x^{\frac{1}{v_1}}$ in Theorem \ref{thm3.1}.
\end{remark}

For $x>0$, let
$$h_n(x)=\frac{n}{v_1}\alpha_nx^{\frac{1}{v_1}-1}F^{n-1}\left(\alpha_n|x|^{\beta_n}\sign(x)\right)f\left(\alpha_n|x|^{\beta_n}\sign(x)\right)$$
denote the pdf of the $M_n$ under power normalization. Then
\begin{eqnarray}\label{eq3.14}
\Delta_n^p(h_n,\Phi_{1}^{'};x)&=&h_n(x)-\Phi_{1}^{'}(x) \nonumber \\
&=&\frac{x^{\frac{1}{v_1}-1}}{v_1}\left(na_nF^{n-1}\left(a_nx^{\frac{1}{v_1}}
\right)f\left(a_nx^{\frac{1}{v_1}}\right)-v_1x^{-1-\frac{1}{v_1}}\Phi_{v_1}(x^{\frac{1}{v_1}})\right) \nonumber\\
&=&\frac{x^{\frac{1}{v_1}-1}}{v_1}\Delta_n(g_n,\Phi_{v_1}^{'};x^{\frac{1}{v_1}}).
\end{eqnarray}
for $x>0$, where $\Phi_1^{'}(x)=x^{-2}\Phi_1(x)=\frac{x^{\frac{1}{v_1}-1}}{v_1}\Phi_{v_1}^{'}\left(x^{\frac{1}{v_1}}\right)$. By using Theorem \ref{thm3.2},  we can derive the higher-order expansions of the pdf of $M_n$ under
power normalization stated as follows.

\begin{remark}\label{rem3.2}
Note that \eqref{eq3.14} shows that
\begin{eqnarray*}
\Delta_n^p(h_n,\Phi_{1}^{'}(x);x)=\frac{x^{\frac{1}{v_1}-1}}{v_1}
\Delta_n(g_n,\Phi_{v_1}^{'};x^{\frac{1}{v_1}})
\end{eqnarray*}
holds with normalized constants $\alpha_n=a_n$ and $\beta_n=\frac{1}{v_1}$, where
$a_n$ is given by \eqref{eq2.2}. Hence, the higher order expansions of the pdf of the extremes from the mixed skew-t sample under power normalization can be calculated straightly by using \eqref{eq3.14} and Theorem \ref{thm3.2}.
\end{remark}

\section{Numerical analysis}\label{sec4}

In this section, numerical studies are presented to illustrate the accuracy of higher-order expansions
of the cdf and the pdf of $M_n$ under the linear and power normalization.
Let $L_i^{l}(x)$ and $U_i^{l}(x)$, $i=1,2,3$, denote the first-order, the second-order and the third-order asymptotics of the cdf and the pdf of $M_n$ under linear normalization, respectively. Similarly, let $L_i^{p}(x)$ and $U_i^{p}(x)$, $i=1,2,3$, denote
the first-order, the second-order and the third-order asymptotics of the cdf and the pdf
of $M_n$ under power normalization, respectively. Note that the second and the third-order asymptotics are related to the sample size $n$.

To compare the accuracy of actual values with its asymptotics, for fixed $x>0$ let
\begin{eqnarray*}
&&\Delta_i^{l}(x)=|F^n(a_nx)-L_i^{l}(x)|,\\
&&\delta_i^{l}(x)=|g_n(x)-U_i^{l}(x)|
\end{eqnarray*}
denote the absolute errors of the cdf and the pdf under two
normalization, where $i=1,2,3$. From Remark \ref{rem3.1} and \ref{rem3.2}, it follows
that $L_i^p(x)=L_{i}^{l}\left( x^{\frac{1}{v_1}} \right)$, $U_i^{p}(x)=\frac{x^{\frac{1}{v_1}-1}}{v_1}U_{i}^{l}\left( x^{\frac{1}{v_1}} \right)$
for $x>0$, where $i=1,2,3$. Then the absolute errors of the cdf and the pdf under power normalization are given by
\begin{eqnarray*}
&&\Delta_i^{p}(x)=|F^n(\alpha_nx^{\beta_n})-L_i^{p}(x)|=\Delta_i^{l}\left(x^{\frac{1}{v_1}}\right),\\
&&\delta_i^{p}(x)=\left|h_n(x)-U_i^{p}(x)\right|
=\frac{x^{\frac{1}{v_1}-1}}{v_1}\delta_i^{l}\left(x^{\frac{1}{v_1}}\right)
\end{eqnarray*}
for $x>0$, where $i=1,2,3$.
We use MATLAB to calculate the asymptotics and the actual values of the
cdf and the pdf of $M_n$ under two different normalization in the following two examples,
where Example 1 focuses on
the cdf of $M_n$, and Example 2 is related to the pdf of $M_n$.

{\bf Example 1.}
Let $X_1\sim \ST_2(1)$, $X_2\sim \ST_3(1.5)$, $X_3\sim \ST_4(2)$, and $T'$ is defined by
\begin{eqnarray}\label{exp1}
T'=\left\{
\begin{aligned}
&X_1, & \quad \P(T'=X_1)=0.5,& \\
&X_2, & \quad \P(T'=X_2)=0.3,& \\
&X_3, & \quad \P(T'=X_3)=0.2.& \\
\end{aligned}
\right.
\end{eqnarray}
Let $M_{n}=\max_{1\le k\le n}\{T_k\}$ with $T_{k}\overset{d}{=}T'$, $k=1,\dots,n$. From Theorem \ref{thm3.1} (iii) and Remark \ref{rem3.1}, we can get
the asymptotics of the cdf of $M_{n}$ as follows:
\begin{eqnarray}\label{as1}
\left\{
\begin{aligned}
&L_1^{l}(x)=\Phi_2(x),\\
&L_2^{l}(x)=\left(1-\frac{9C_3T_4(3)}{5\sqrt{2}C_2T_3(\sqrt{3})}x^{-3}a_n^{-1}\right)\Phi_2(x),\\
&L_3^{l}(x)=\left[1-\frac{9C_3T_4(3)}{5\sqrt{2}C_2T_3(\sqrt{3})}x^{-3}a_n^{-1}
-\left(  -\left(\frac{\sqrt{3}C_3}{8T_3(\sqrt{3})}+1.5 \right)x^{-4}-\frac{81}{100}\left( \frac{C_3T_4(3)}{C_2T_3(\sqrt{3})} \right)^2 x^{-6}  \right. \right. \\
&\quad \quad \quad \quad  \left. \left. +
\frac{T_3(\sqrt{3})}{4}x^{-4}+\frac{16C_4T_5(2\sqrt{5})}{5\sqrt{2}C_2T_3(\sqrt{3})}x^{-4}\right)a_n^{-2}\right]\Phi_2(x),\\
&L_i^{p}(x)=L_i^{l}(x^{\frac{1}{2}}), \quad i=1,2,3.
\end{aligned}
\right.
\end{eqnarray}

First we calculate the absolute errors of the cdf of $M_n$ at $x=2$ and $x=0.7$ for $n$ varying from $25$ to $1000$ with lattice $25$. For $x=2$ and $x=0.7$, numerical analysis results of $\Delta_i^l(x)$ and $\Delta_i^p(x)$, $i=1,2,3$, are documented in Tables \ref{table:1}-\ref{table:2}.
The two tables show that accuracies of all three kinds of asymptotics of the cdf
can be improved as $n$ becomes large.

In order to show the accuracy of all asymptotics more intuitive with
varying $n$ , we then plot the actual values and its asymptotics of the cdf
of $M_n$ with fixed $x$. With $x=2$, Figure \ref{fig:1} compares all asymptotics with the actual value under
two different normalization; while Figure \ref{fig:2} shows the case of $x=0.7$. Tables \ref{table:1}-\ref{table:2} and Figures \ref{fig:1}-\ref{fig:2} show the following facts:
i) For large $n$, the third-order asymptotics of the cdf of $M_n$ are closer to the actual values
under the two different normalization. ii) For large $n$, $ \Delta_3^{l}(2)$ is smaller than
$ \Delta_3^{p}(2)$, which shows that the third-order asymptotics of the cdf of $M_n$
at $x=2$ are more closer to its actual value under linear
normalization. iii) For large $n$, $\Delta_3^{l}(0.7)$
is larger than $\Delta_3^{p}(0.7)$, which shows that the third-order
asymptotic of the cdf of $M_n$ at $x=0.7$ are more closer to its actual value under power normalization.

{\bf Example 2.} Let $X_1\sim \ST_3(1)$, $X_2\sim \ST_6(2)$, $X_3\sim \ST_8(3)$, and $T''$ is defined by
\begin{eqnarray}\label{exp2}
T''=\left\{
\begin{aligned}
&X_1, & \quad \P(T''=X_1)=0.5,& \\
&X_2, & \quad \P(T''=X_2)=0.3,& \\
&X_3, & \quad \P(T''=X_3)=0.2.& \\
\end{aligned}
\right.
\end{eqnarray}
Let $M_{n}=\max_{1\le k\le n}\{T_k\}$ with $T_{k}\overset{d}{=}T''$, $k=1,\dots,n$. By using Theorem \ref{thm3.2} (ii) and Remark \ref{rem3.2}, the asymptotics of the pdf of $M_{n}$ are given
\begin{eqnarray}\label{as2}
\left\{
\begin{aligned}
&U_1^{l}(x)=\Phi_3^{'}(x),\\
&U_2^{l}(x)=\left(1+\left( -\frac{3(4T_4(2)+2^{-1.5}C_4)}{2T_4(2)}x^{-2}
+\left( \frac{9C_4}{20\sqrt{2}T_4(2)}+3.6\right)x^{-5}\right)a_n^{-2} \right)\Phi_3^{'}(x),\\
&U_3^{l}(x)=\left[1+\left( -\frac{3(4T_4(2)+2^{-1.5}C_4)}{2T_4(2)}x^{-2}
+\left( \frac{9C_4}{20\sqrt{2}T_4(2)}+3.6\right)x^{-5}\right)a_n^{-2} \right. \\
& \quad  \quad  \quad  \quad \left.
+\left((2x^{-3}-x^{-6})\frac{36\sqrt{6}C_6T_7(2\sqrt{7})}{5C_3T_4(2)}+\frac{\sqrt{3}}{\pi}T_{4}(2)(2-x^{-3})x^{-3}\right)a_n^{-3} \right]\Phi_3^{'}(x),\\
&U_i^{p}(x)=\frac{x^{-\frac{2}{3}}}{3}U^{l}(x^{\frac{1}{3}}), \quad i=1,2,3.
\end{aligned}
\right.
\end{eqnarray}

Here, we calculate the absolute errors of the pdf of $M_n$ at $x=3$  for $n$ varying from $25$ to $1000$ with lattice $25$,
and at $x=0.75$ for $n$ varying from $375$ to $15000$ with lattice $375$. Tables \ref{table:3}-\ref{table:4} document the numerical analysis results of $\delta_i^l(x)$ and $\delta_i^p(x)$, $i=1,2,3$, which show that
the accuracy of all three kinds of asymptotics of pdf
improve as $n$ becomes large.
Figures \ref{fig:3}-\ref{fig:4} compare all asymptotics with the
actual values under two different normalizations.
From Tables \ref{table:3}-\ref{table:4} and Figures \ref{fig:3}-\ref{fig:4}, we know that:
i) For large $n$, the third-order asymptotics of pdf of $M_n$ are closer to the actual values
under the two different normalization. ii) When $n$ is larger, $\delta_3^{l}(3)$ is smaller than
$\delta_3^{p}(3)$, which shows that the third-order asymptotic of the pdf of $M_n$  at $x=3$ are more closer to its actual value under linear
normalization. iii) For large $n$, $\delta_3^{l}(0.75)$
is larger than $\delta_3^{p}(0.75)$, which shows that the third-order
asymptotic of the pdf of $M_n$ at $x=0.75$ are more closer to its actual value under power normalization.

\begin{table}\centering
\caption{\c{Absolute errors between actual values and their asymptotics of the cdf at $x=2$}}\label{table:1}
\begin{tabular}{ccccccc}
\hline
$n$ & $\Delta_1^l(2)$ & $\Delta_1^p(2)$ & $\Delta_2^l(2)$ & $\Delta_{2}^{p}(2)$ & $\Delta_3^l(2)$ & $\Delta_3^p(2)$\\
\hline
25&	0.041618314&	0.083694396&	0.000428915&	0.007036842&	0.003937814&	 0.017246955\\
50&	0.029959447&	0.063092667&	0.000834143&	0.001064006&	0.001349221&	 0.006169063\\
75&	0.024516994&	0.052450469&	0.000736283&	0.000066765&	0.000719293&	 0.003336607\\
100&	0.021225874&	0.045767313&	0.000631175&	0.000401695&	0.000460507&	 0.002150834\\
125&	0.018967784&	0.041089738&	0.000547325&	0.000513495&	0.000326021&	 0.001528528\\
150&	0.017297315&	0.037586625&	0.000481813&	0.000545752&	0.000245975&	 0.001155933\\
175&	0.015998039&	0.034839136&	0.000429909&	0.000545952&	0.000193909&	 0.000912636\\
200&	0.014950642&	0.032610918&	0.000387990&	0.000532581&	0.000157851&	 0.000743683\\
225&	0.014083308&	0.030757365&	0.000353508&	0.000513619&	0.000131684&	 0.000620838\\
250&	0.013349908&	0.029184474&	0.000324676&	0.000492737&	0.000111997&	 0.000528274\\
275&	0.012719296&	0.027828177&	0.000300225&	0.000471680&	0.000096751&	 0.000456512\\
300&	0.012169587&	0.026643128&	0.000279231&	0.000451276&	0.000084663&	 0.000399567\\
325&	0.011684897&	0.025596216&	0.000261013&	0.000431899&	0.000074890&	 0.000353495\\
350&	0.011253383&	0.024662633&	0.000245053&	0.000413690&	0.000066856&	 0.000315604\\
375&	0.010866013&	0.023823378&	0.000230956&	0.000396674&	0.000060159&	 0.000284001\\
400&	0.010515764&	0.023063622&	0.000218414&	0.000380812&	0.000054507&	 0.000257320\\
425&	0.010197078&	0.022371599&	0.000207181&	0.000366042&	0.000049685&	 0.000234553\\
450&	0.009905498&	0.021737842&	0.000197063&	0.000352285&	0.000045533&	 0.000214944\\
475&	0.009637397&	0.021154635&	0.000187900&	0.000339462&	0.000041927&	 0.000197913\\
500&	0.009389793&	0.020615617&	0.000179563&	0.000327496&	0.000038773&	 0.000183010\\
525&	0.009160209&	0.020115494&	0.000171945&	0.000316314&	0.000035994&	 0.000169881\\
550&	0.008946565&	0.019649815&	0.000164956&	0.000305850&	0.000033532&	 0.000158246\\
575&	0.008747104&	0.019214812&	0.000158520&	0.000296041&	0.000031338&	 0.000147877\\
600&	0.008560326&	0.018807269&	0.000152575&	0.000286832&	0.000029372&	 0.000138589\\
625&	0.008384945&	0.018424420&	0.000147065&	0.000278173&	0.000027604&	 0.000130232\\
650&	0.008219851&	0.018063876&	0.000141945&	0.000270017&	0.000026006&	 0.000122680\\
675&	0.008064078&	0.017723558&	0.000137174&	0.000262324&	0.000024557&	 0.000115829\\
700&	0.007916783&	0.017401648&	0.000132718&	0.000255056&	0.000023237&	 0.000109591\\
725&	0.007777225&	0.017096549&	0.000128546&	0.000248181&	0.000022031&	 0.000103892\\
750&	0.007644753&	0.016806849&	0.000124632&	0.000241668&	0.000020926&	 0.000098670\\
775&	0.007518787&	0.016531300&	0.000120952&	0.000235489&	0.000019910&	 0.000093869\\
800&	0.007398812&	0.016268789&	0.000117486&	0.000229620&	0.000018974&	 0.000089446\\
825&	0.007284370&	0.016018321&	0.000114216&	0.000224039&	0.000018109&	 0.000085358\\
850&	0.007175049&	0.015779004&	0.000111125&	0.000218726&	0.000017308&	 0.000081572\\
875&	0.007070479&	0.015550039&	0.000108200&	0.000213660&	0.000016564&	 0.000078057\\
900&	0.006970326&	0.015330700&	0.000105426&	0.000208827&	0.000015872&	 0.000074788\\
925&	0.006874288&	0.015120333&	0.000102793&	0.000204210&	0.000015227&	 0.000071740\\
950&	0.006782093&	0.014918345&	0.000100290&	0.000199794&	0.000014624&	 0.000068893\\
975&	0.006693489&	0.014724194&	0.000097907&	0.000195569&	0.000014060&	 0.000066229\\
1000&	0.006608252&	0.014537388&	0.000095636&	0.000191520&	0.000013532&	 0.000063733\\
\hline
\end{tabular}
\end{table}

\begin{table}\centering
\caption{\c{Absolute errors between actual values and their asymptotics of the cdf  at $x=0.7$}}
\label{table:2}
\begin{tabular}{ccccccc}
\hline
$n$ & $\Delta_1^l(0.7)$ & $\Delta_1^p(0.7)$ & $\Delta_2^l(0.7)$ & $\Delta_{2}^{p}(0.7)$ & $\Delta_3^l(0.7)$ & $\Delta_3^p(0.7)$\\
\hline
25&	0.072769740&	0.103763098&	0.087495626&	0.069370486&	0.049301919&	 0.061652134\\
50&	0.069353107&	0.091767994&	0.043971619&	0.030655938&	0.024874766&	 0.026796762\\
75&	0.064006392&	0.081576059&	0.028522860&	0.018382662&	0.015791625&	 0.015809878\\
100&	0.059391318&	0.073922726&	0.020741365&	0.012644066&	0.011192938&	 0.010714478\\
125&	0.055557905&	0.068015416&	0.016114945&	0.009412276&	0.008476203&	 0.007868606\\
150&	0.052351600&	0.063304069&	0.013076462&	0.007377421&	0.006710844&	 0.006091029\\
175&	0.049632109&	0.059441782&	0.010942505&	0.005996562&	0.005486261&	 0.004893940\\
200&	0.047292939&	0.056204365&	0.009369424&	0.005007601&	0.004595211&	 0.004042807\\
225&	0.045255350&	0.053441320&	0.008166438&	0.004269875&	0.003922693&	 0.003412280\\
250&	0.043460780&	0.051047909&	0.007219578&	0.003701737&	0.003400208&	 0.002929902\\
275&	0.041865033&	0.048948898&	0.006456793&	0.003252842&	0.002984637&	 0.002551173\\
300&	0.040434206&	0.047088752&	0.005830420&	0.002890609&	0.002647611&	 0.002247413\\
325&	0.039141869&	0.045425507&	0.005307745&	0.002593110&	0.002369768&	 0.001999390\\
350&	0.037967110&	0.043926805&	0.004865610&	0.002345092&	0.002137488&	 0.001793781\\
375&	0.036893155&	0.042567252&	0.004487185&	0.002135647&	0.001940937&	 0.001621090\\
400&	0.035906379&	0.041326607&	0.004159962&	0.001956789&	0.001772856&	 0.001474392\\
425&	0.034995595&	0.040188516&	0.003874466&	0.001802546&	0.001627777&	 0.001348526\\
450&	0.034151520&	0.039139604&	0.003623389&	0.001668373&	0.001501516&	 0.001239576\\
475&	0.033366381&	0.038168814&	0.003401015&	0.001550754&	0.001390820&	 0.001144525\\
500&	0.032633615&	0.037266914&	0.003202810&	0.001446932&	0.001293125&	 0.001061014\\
525&	0.031947636&	0.036426131&	0.003025133&	0.001354715&	0.001206385&	 0.000987174\\
550&	0.031303659&	0.035639862&	0.002865031&	0.001272342&	0.001128954&	 0.000921508\\
575&	0.030697556&	0.034902463&	0.002720083&	0.001198385&	0.001059487&	 0.000862804\\
600&	0.030125744&	0.034209075&	0.002588286&	0.001131670&	0.000996882&	 0.000810072\\
625&	0.029585100&	0.033555488&	0.002467973&	0.001071229&	0.000940225&	 0.000762495\\
650&	0.029072883&	0.032938036&	0.002357741&	0.001016254&	0.000888752&	 0.000719394\\
675&	0.028586680&	0.032353506&	0.002256404&	0.000966067&	0.000841822&	 0.000680203\\
700&	0.028124353&	0.031799075&	0.002162954&	0.000920097&	0.000798893&	 0.000644441\\
725&	0.027684005&	0.031272245&	0.002076526&	0.000877855&	0.000759502&	 0.000611705\\
750&	0.027263942&	0.030770798&	0.001996376&	0.000838925&	0.000723253&	 0.000581647\\
775&	0.026862649&	0.030292760&	0.001921860&	0.000802949&	0.000689805&	 0.000553970\\
800&	0.026478763&	0.029836365&	0.001852419&	0.000769618&	0.000658865&	 0.000528419\\
825&	0.026111057&	0.029400026&	0.001787562&	0.000738662&	0.000630177&	 0.000504773\\
850&	0.025758421&	0.028982317&	0.001726862&	0.000709848&	0.000603518&	 0.000482837\\
875&	0.025419852&	0.028581948&	0.001669939&	0.000682969&	0.000578690&	 0.000462445\\
900&	0.025094433&	0.028197751&	0.001616461&	0.000657847&	0.000555525&	 0.000443448\\
925&	0.024781333&	0.027828666&	0.001566130&	0.000634320&	0.000533868&	 0.000425716\\
950&	0.024479791&	0.027473727&	0.001518684&	0.000612249&	0.000513586&	 0.000409135\\
975&	0.024189111&	0.027132053&	0.001473886&	0.000591509&	0.000494560&	 0.000393602\\
1000&	0.023908653&	0.026802837&	0.001431526&	0.000571986&	0.000476683&	 0.000379028\\
\hline
\end{tabular}
\end{table}

\begin{table}\centering
\caption{\c{Absolute errors between actual values and their asymptotics of the pdf at $x=3$}}
\label{table:3}
\begin{tabular}{ccccccc}
\hline
$n$ & $\delta_1^l(3)$ & $\delta_1^p(3)$ & $\delta_2^l(3)$ & $\delta_2^{p}(3)$ & $\delta_3^l(3)$ & $\delta_3^p(3)$\\
\hline
25&	0.0008041273&	0.0006725280&	0.0019423949&	0.0223613528&	0.0000566926&	 0.0117149407\\
50&	0.0007311965&	0.0000831380&	0.0009990041&	0.0135799655&	0.0000005396&	 0.0034581813\\
75&	0.0006485578&	0.0006433885&	0.0006718323&	0.0097835109&	0.0000054698&	 0.0015752536\\
100&	0.0005840952&	0.0009576983&	0.0005058628&	0.0076495175&	0.0000060910&	 0.0008695559\\
125&	0.0005337537&	0.0011375344&	0.0004055436&	0.0062799394&	0.0000057262&	 0.0005353193\\
150&	0.0004934233&	0.0012425557&	0.0003383704&	0.0053259793&	0.0000051892&	 0.0003534029\\
175&	0.0004603054&	0.0013037635&	0.0002902528&	0.0046232692&	0.0000046689&	 0.0002447727\\
200&	0.0004325366&	0.0013381242&	0.0002540939&	0.0040840820&	0.0000042080&	 0.0001754547\\
225&	0.0004088465&	0.0013554372&	0.0002259309&	0.0036572932&	0.0000038100&	 0.0001289616\\
250&	0.0003883432&	0.0013616312&	0.0002033771&	0.0033110845&	0.0000034683&	 0.0000965449\\
275&	0.0003703823&	0.0013604338&	0.0001849095&	0.0030246130&	0.0000031743&	 0.0000732319\\
300&	0.0003544865&	0.0013542639&	0.0001695107&	0.0027836538&	0.0000029201&	 0.0000560373\\
325&	0.0003402938&	0.0013447330&	0.0001564749&	0.0025781659&	0.0000026989&	 0.0000430875\\
350&	0.0003275248&	0.0013329376&	0.0001452973&	0.0024008590&	0.0000025053&	 0.0000331620\\
375&	0.0003159599&	0.0013196371&	0.0001356072&	0.0022463123&	0.0000023347&	 0.0000254406\\
400&	0.0003054237&	0.0013053648&	0.0001271265&	0.0021104110&	0.0000021835&	 0.0000193573\\
425&	0.0002957744&	0.0012904994&	0.0001196421&	0.0019899757&	0.0000020488&	 0.0000145122\\
450&	0.0002868962&	0.0012753118&	0.0001129885&	0.0018825105&	0.0000019281&	 0.0000106170\\
475&	0.0002786930&	0.0012599970&	0.0001070346&	0.0017860291&	0.0000018194&	 0.0000074600\\
500&	0.0002710848&	0.0012446954&	0.0001016756&	0.0016989310&	0.0000017212&	 0.0000048837\\
525&	0.0002640041&	0.0012295082&	0.0000968267&	0.0016199119&	0.0000016320&	 0.0000027687\\
550&	0.0002573936&	0.0012145077&	0.0000924184&	0.0015478987&	0.0000015508&	 0.0000010237\\
575&	0.0002512043&	0.0011997449&	0.0000883933&	0.0014820000&	0.0000014764&	 0.0000004220\\
600&	0.0002453940&	0.0011852554&	0.0000847035&	0.0014214695&	0.0000014082&	 0.0000016239\\
625&	0.0002399263&	0.0011710628&	0.0000813089&	0.0013656775&	0.0000013454&	 0.0000026258\\
650&	0.0002347692&	0.0011571824&	0.0000781755&	0.0013140891&	0.0000012875&	 0.0000034624\\
675&	0.0002298950&	0.0011436225&	0.0000752741&	0.0012662469&	0.0000012338&	 0.0000041620\\
700&	0.0002252792&	0.0011303870&	0.0000725800&	0.0012217575&	0.0000011840&	 0.0000047470\\
725&	0.0002209002&	0.0011174758&	0.0000700718&	0.0011802808&	0.0000011377&	 0.0000052362\\
750&	0.0002167386&	0.0011048864&	0.0000677308&	0.0011415210&	0.0000010946&	 0.0000056445\\
775&	0.0002127775&	0.0010926142&	0.0000655409&	0.0011052198&	0.0000010542&	 0.0000059846\\
800&	0.0002090015&	0.0010806531&	0.0000634880&	0.0010711509&	0.0000010165&	 0.0000062667\\
825&	0.0002053970&	0.0010689961&	0.0000615595&	0.0010391145&	0.0000009811&	 0.0000064996\\
850&	0.0002019515&	0.0010576355&	0.0000597445&	0.0010089343&	0.0000009478&	 0.0000066904\\
875&	0.0001986540&	0.0010465630&	0.0000580333&	0.0009804536&	0.0000009166&	 0.0000068452\\
900&	0.0001954943&	0.0010357704&	0.0000564173&	0.0009535330&	0.0000008871&	 0.0000069693\\
925&	0.0001924633&	0.0010252489&	0.0000548886&	0.0009280477&	0.0000008592&	 0.0000070668\\
950&	0.0001895527&	0.0010149901&	0.0000534405&	0.0009038861&	0.0000008329&	 0.0000071416\\
975&	0.0001867548&	0.0010049853&	0.0000520667&	0.0008809478&	0.0000008080&	 0.0000071967\\
1000&	0.0001840628&	0.0009952262&	0.0000507616&	0.0008591422&	0.0000007844&	 0.0000072349\\
\hline
\end{tabular}
\end{table}
\begin{table}\centering
\caption{\c{Absolute errors between actual values and their asymptotics of the pdf at $x=0.75$}}
\label{table:4}
\begin{tabular}{ccccccc}
\hline
$n$ & $\delta_1^l(0.75)$ & $\delta_1^p(0.75)$ & $\delta_2^l(0.75)$ & $\delta_2^{p}(0.75)$ & $\delta_3^l(0.75)$ & $\delta_3^p(0.75)$\\
\hline
375&	0.029225833&	0.004206893&	0.048225582&	0.017429407&	 0.008274292&	 0.003965312\\
750&	0.026010161&	0.001357781&	0.022781173&	0.009687443&	 0.002805528&	 0.001009916\\
1125&	0.022436305&	0.000326059&	0.014798452&	0.006682780&	 0.001481355&	 0.000448792\\
1500&	0.019788382&	0.000153010&	0.010948233&	0.005094348&	 0.000960410&	 0.000254332\\
1875&	0.017800136&	0.000408380&	0.008687875&	0.004113656&	 0.000697617&	 0.000165288\\
2250&	0.016254384&	0.000555948&	0.007202043&	0.003448535&	 0.000543494&	 0.000117251\\
2625&	0.015014707&	0.000645367&	0.006150896&	0.002968027&	 0.000443569&	 0.000088361\\
3000&	0.013994922&	0.000700876&	0.005367932&	0.002604753&	 0.000374021&	 0.000069587\\
3375&	0.013138561&	0.000735459&	0.004762043&	0.002320534&	 0.000323010&	 0.000056657\\
3750&	0.012407202&	0.000756573&	0.004279199&	0.002092132&	 0.000284070&	 0.000047340\\
4125&	0.011773795&	0.000768735&	0.003885333&	0.001904593&	 0.000253397&	 0.000040382\\
4500&	0.011218727&	0.000774802&	0.003557896&	0.001747864&	 0.000228621&	 0.000035029\\
4875&	0.010727415&	0.000776643&	0.003281370&	0.001614938&	 0.000208194&	 0.000030810\\
5250&	0.010288769&	0.000775517&	0.003044726&	0.001500779&	 0.000191062&	 0.000027415\\
5625&	0.009894201&	0.000772289&	0.002839907&	0.001401679&	 0.000176487&	 0.000024636\\
6000&	0.009536942&	0.000767570&	0.002660891&	0.001314845&	 0.000163936&	 0.000022325\\
6375&	0.009211582&	0.000761796&	0.002503088&	0.001238134&	 0.000153013&	 0.000020379\\
6750&	0.008913739&	0.000755280&	0.002362935&	0.001169875&	 0.000143419&	 0.000018720\\
7125&	0.008639821&	0.000748254&	0.002237626&	0.001108745&	 0.000134926&	 0.000017293\\
7500&	0.008386855&	0.000740888&	0.002124919&	0.001053683&	 0.000127354&	 0.000016053\\
7875&	0.008152358&	0.000733310&	0.002023003&	0.001003829&	 0.000120561&	 0.000014967\\
8250&	0.007934233&	0.000725614&	0.001930400&	0.000958477&	 0.000114432&	 0.000014010\\
8625&	0.007730701&	0.000717871&	0.001845886&	0.000917045&	 0.000108874&	 0.000013160\\
9000&	0.007540241&	0.000710134&	0.001768448&	0.000879046&	 0.000103811&	 0.000012400\\
9375&	0.007361541&	0.000702444&	0.001697231&	0.000844070&	 0.000099180&	 0.000011719\\
9750&	0.007193466&	0.000694831&	0.001631515&	0.000811771&	 0.000094927&	 0.000011103\\
10125&	0.007035027&	0.000687317&	0.001570686&	0.000781852&	 0.000091009&	 0.000010545\\
10500&	0.006885357&	0.000679917&	0.001514218&	0.000754059&	 0.000087387&	 0.000010038\\
10875&	0.006743695&	0.000672644&	0.001461660&	0.000728175&	 0.000084029&	 0.000009574\\
11250&	0.006609368&	0.000665505&	0.001412617&	0.000704009&	 0.000080907&	 0.000009148\\
11625&	0.006481779&	0.000658506&	0.001366749&	0.000681396&	 0.000077998&	 0.000008757\\
12000&	0.006360395&	0.000651650&	0.001323758&	0.000660190&	 0.000075280&	 0.000008395\\
12375&	0.006244742&	0.000644938&	0.001283381&	0.000640264&	 0.000072736&	 0.000008061\\
12750&	0.006134393&	0.000638371&	0.001245387&	0.000621506&	 0.000070349&	 0.000007750\\
13125&	0.006028964&	0.000631947&	0.001209571&	0.000603816&	 0.000068105&	 0.000007462\\
13500&	0.005928109&	0.000625666&	0.001175750&	0.000587105&	 0.000065992&	 0.000007192\\
13875&	0.005831515&	0.000619525&	0.001143764&	0.000571295&	 0.000063999&	 0.000006941\\
14250&	0.005738896&	0.000613522&	0.001113466&	0.000556314&	 0.000062116&	 0.000006705\\
14625&	0.005649995&	0.000607653&	0.001084726&	0.000542099&	 0.000060334&	 0.000006484\\
15000&	0.005564575&	0.000601917&	0.001057428&	0.000528592&	 0.000058646&	 0.000006276\\
\hline
\end{tabular}
\end{table}

\begin{figure}[H]
\centering
\begin{minipage}[t]{0.5\textwidth}
\centering
\includegraphics[height = 6cm, width = 8 cm]{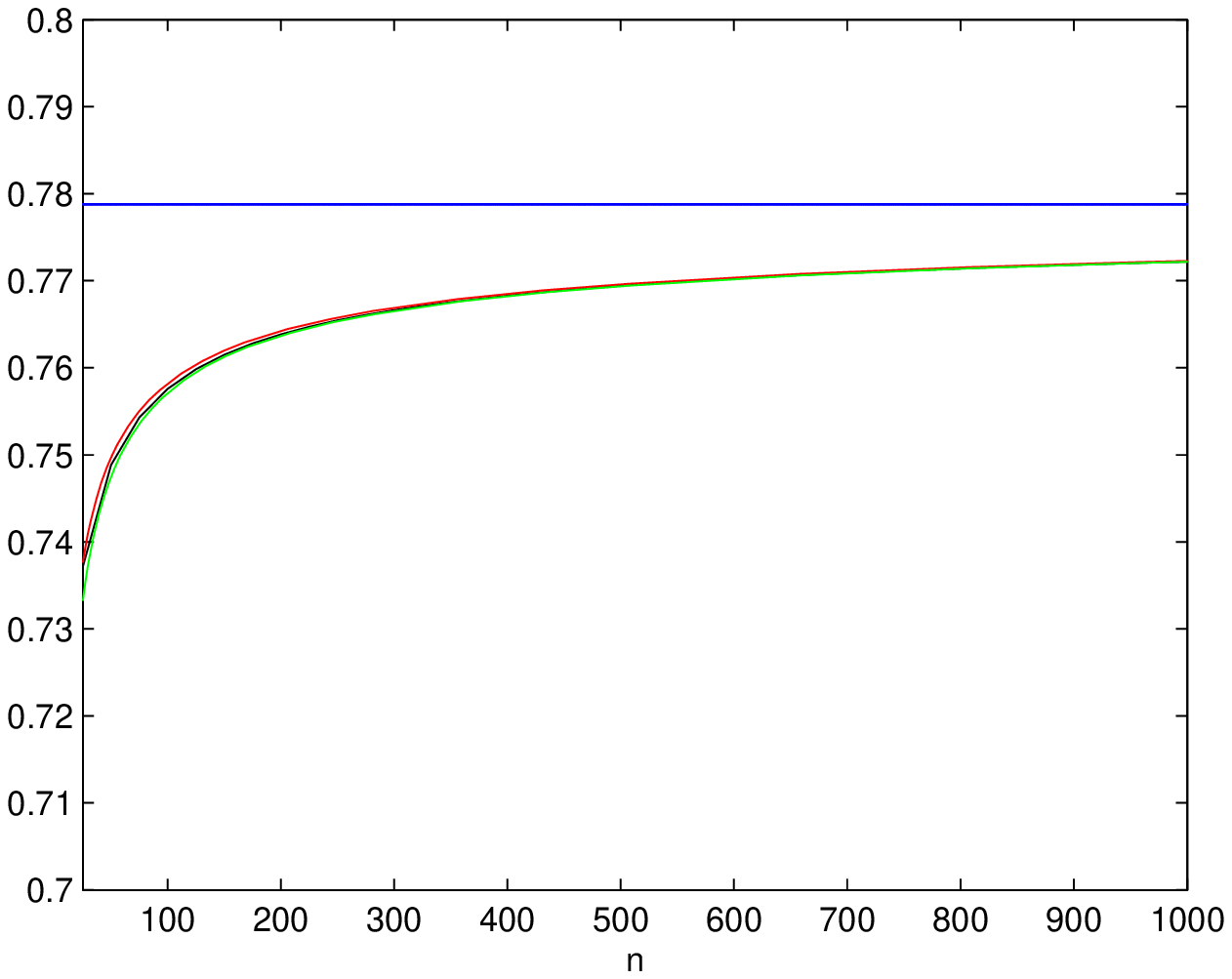}
{(a) under linear normalization}
\end{minipage}
\begin{minipage}[t]{0.45\textwidth}
\centering
\includegraphics[height = 6cm, width = 8 cm]{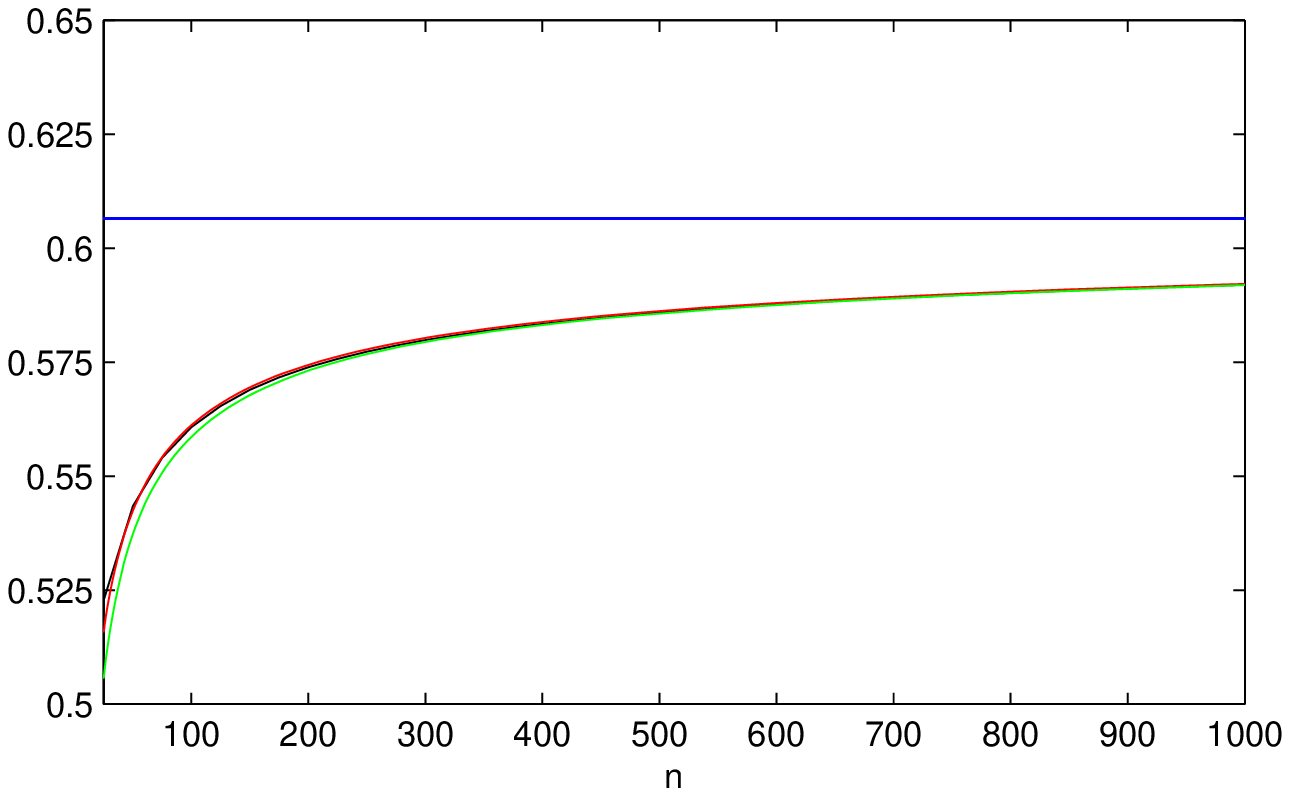}
{(b) under power normalization}
\end{minipage}
\caption{Actual values and its approximations of the cdf of $M_n$ with $x=2$ under linear and power normalization, where
the $\MSTD$ is given by \eqref{exp1}, and all asymptotics are given by \eqref{as1}. The actual
values drawn in black, the first-order asymptotics drawn in blue, the second-order asymptotics
drawn in red and the third-order asymptotics drawn in green.
}
\label{fig:1}
\end{figure}

\begin{figure}[H]\label{fig:2}
\centering
\begin{minipage}[t]{0.5\textwidth}
\centering
\includegraphics[height = 6cm, width = 8 cm]{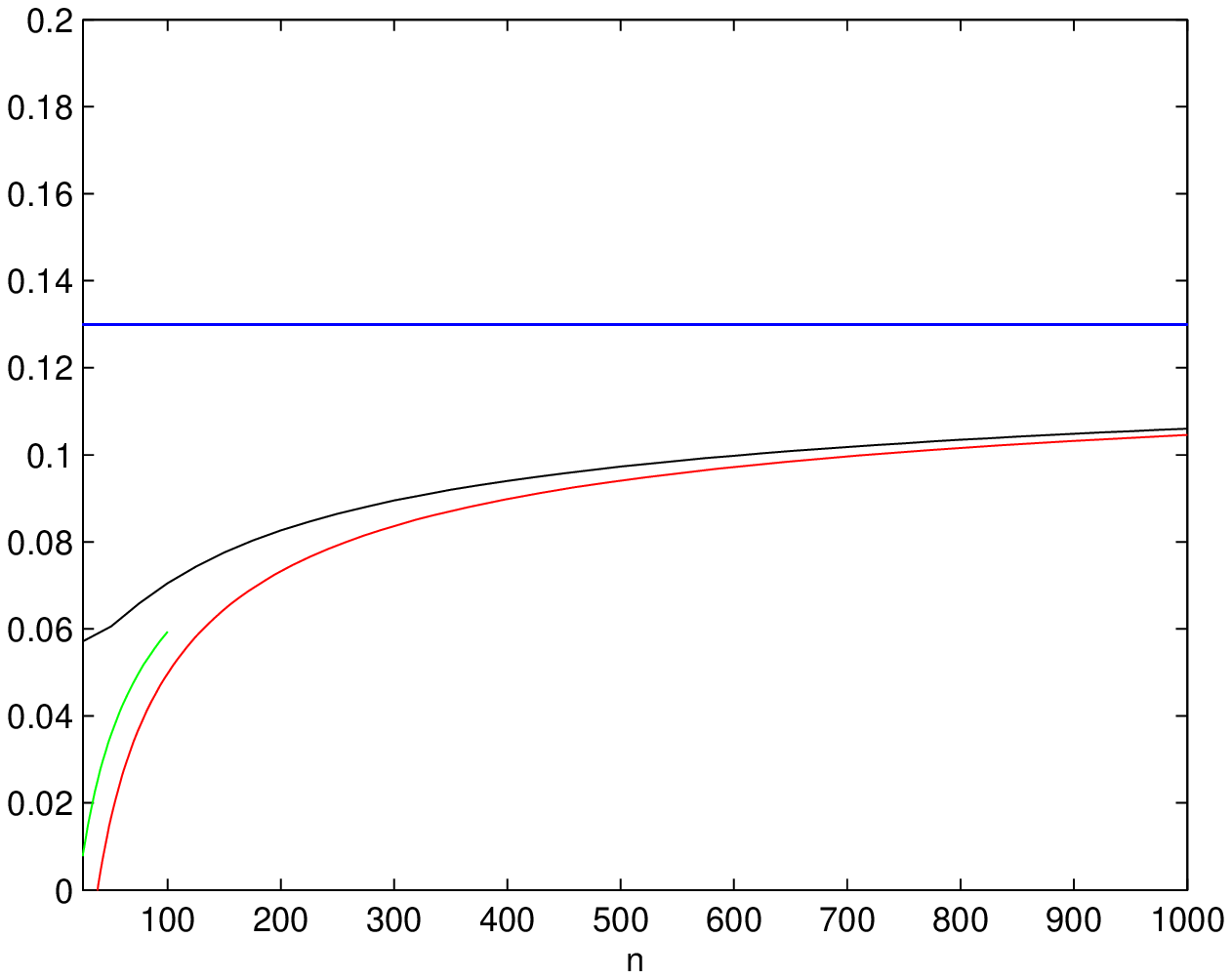}
{(a) under linear normalization}
\end{minipage}
\begin{minipage}[t]{0.45\textwidth}
\centering
\includegraphics[height = 6cm, width = 8 cm]{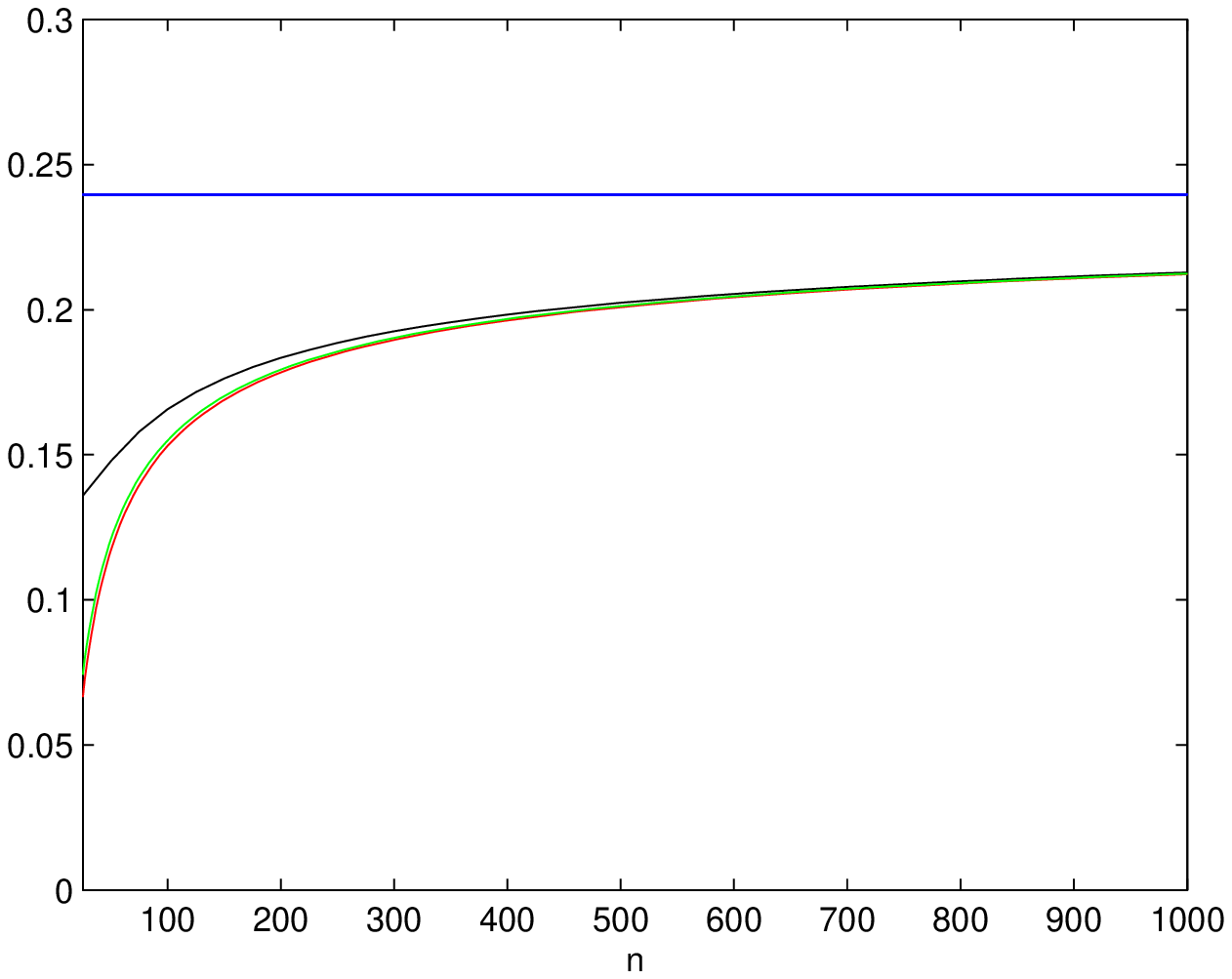}
{(b) under power normalization}
\end{minipage}
\caption{Actual values and its approximations of the cdf of $M_n$ with $x=0.7$ under linear and power normalization, where
the $\MSTD$ is given by \eqref{exp1}, and all asymptotics are given by \eqref{as1}. The actual
values drawn in black, the first-order asymptotics drawn in blue, the second-order asymptotics
drawn in red and the third-order asymptotics drawn in green. }
\label{fig:2}
\end{figure}

\begin{figure}[H]\label{fig:3}
\centering
\begin{minipage}[t]{0.5\textwidth}
\centering
\includegraphics[height = 6cm, width = 8 cm]{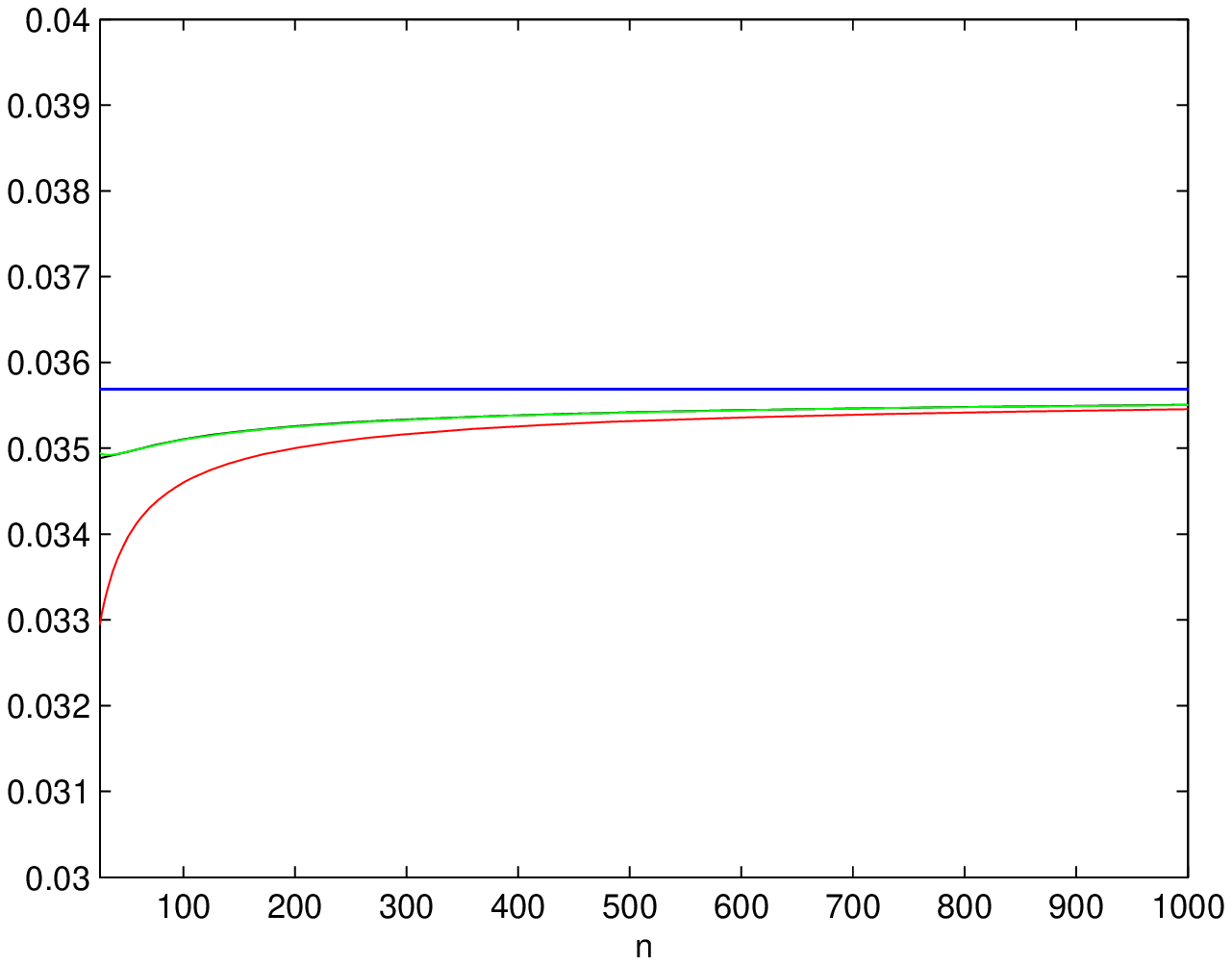}
{(a) under linear normalization}
\end{minipage}
\begin{minipage}[t]{0.45\textwidth}
\centering
\includegraphics[height = 6cm, width = 8 cm]{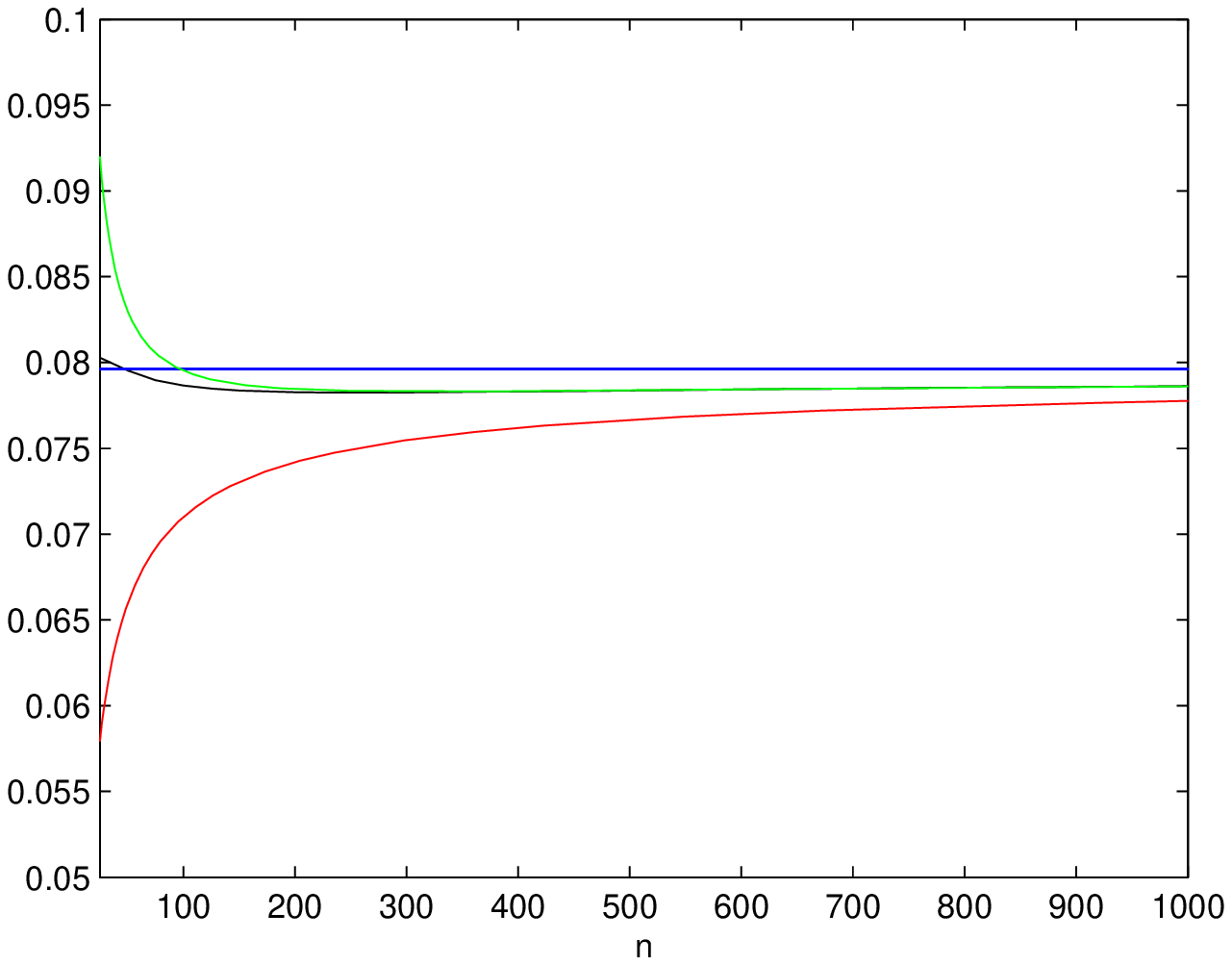}
{(b) under power normalization}
\end{minipage}
\caption{Actual values and its approximations of the pdf of $M_n$ with $x=3$ under linear and power normalization, where
the $\MSTD$ is given by \eqref{exp2}, and all asymptotics are given by \eqref{as2}. The actual
values drawn in black, the first-order asymptotics drawn in blue, the second-order asymptotics
drawn in red and the third-order asymptotics drawn in green. }
\label{fig:3}
\end{figure}

\begin{figure}[H]\label{fig:4}
\centering
\begin{minipage}[t]{0.5\textwidth}
\centering
\includegraphics[height = 6cm, width = 8 cm]{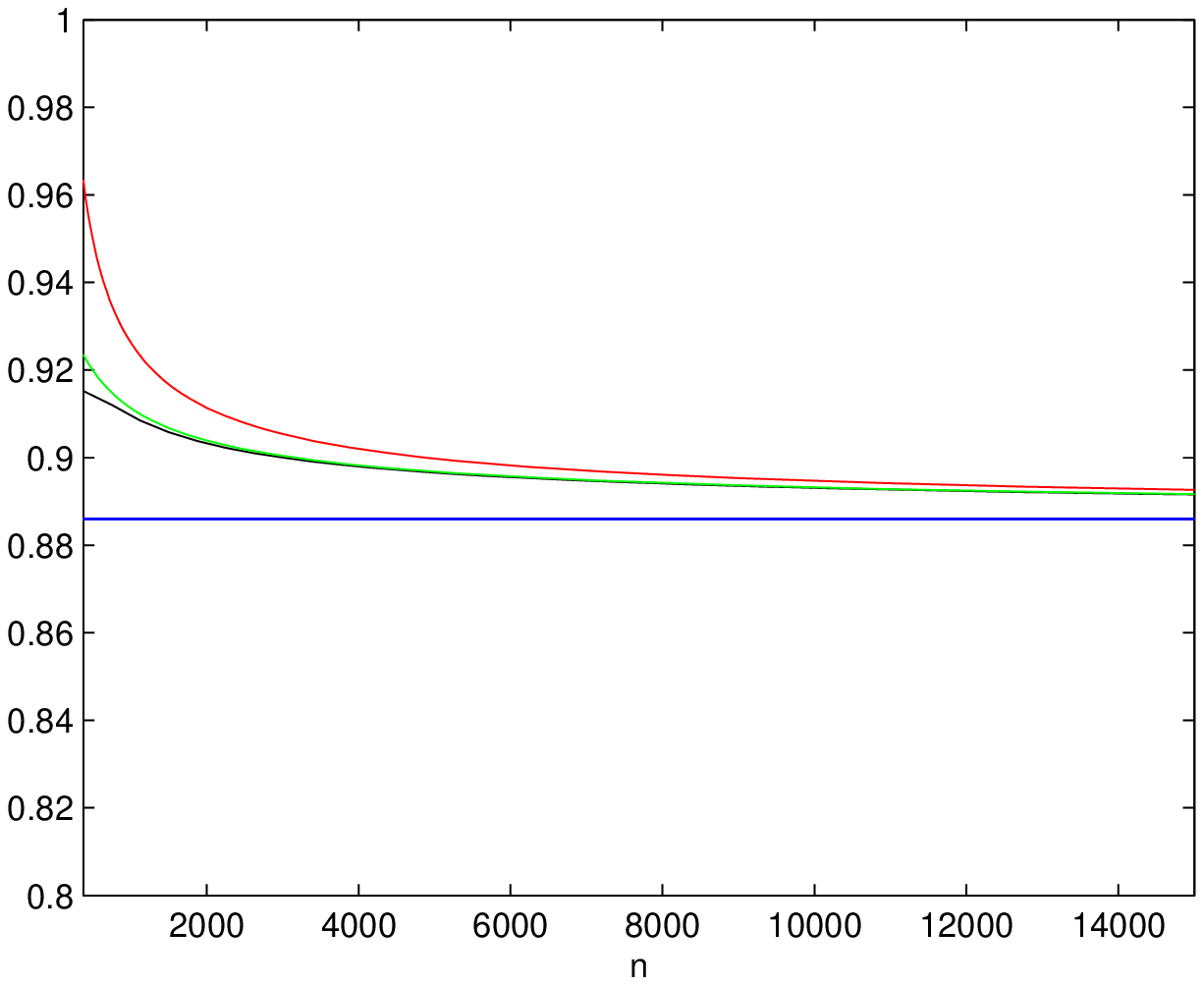}
{(a) under linear normalization}
\end{minipage}
\begin{minipage}[t]{0.45\textwidth}
\centering
\includegraphics[height = 6cm, width = 8 cm]{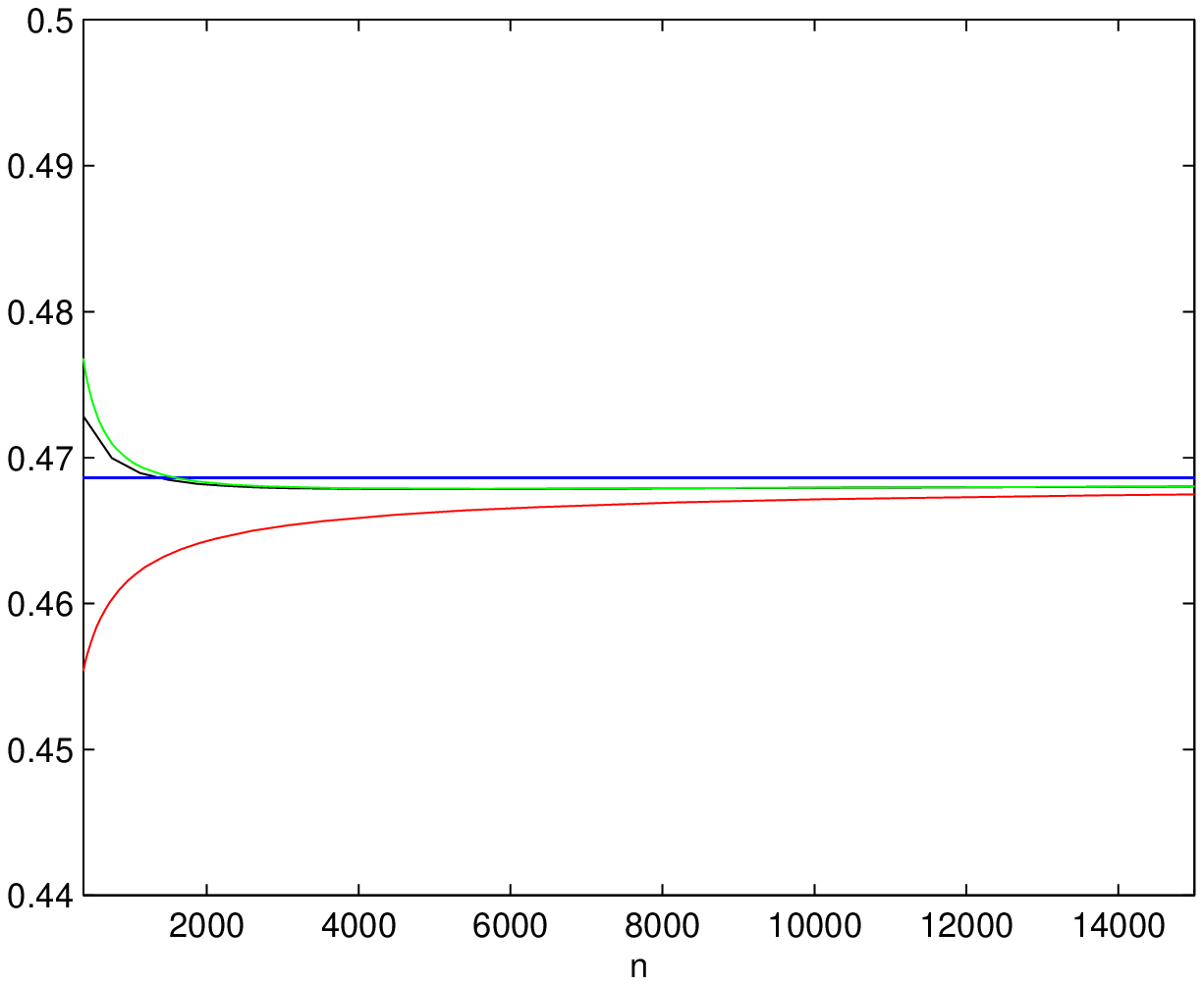}
{(b) under power normalization}
\end{minipage}
\caption{Actual values and its approximations of the pdf of $M_n$ with $x=0.75$ under linear and power normalization, where
the $\MSTD$ is given by \eqref{exp2}, and all asymptotics are given by \eqref{as2}. The actual
values drawn in black, the first-order asymptotics drawn in blue, the second-order asymptotics
drawn in red and the third-order asymptotics drawn in green. }
\label{fig:4}
\end{figure}

\noindent
{\bf Acknowledgements}~~This work was supported by the National Natural Science Foundation
of China Grant No. 11171275, and the Doctoral Grant of University of Shanghai for Science and Technology (BSQD201608).

\end{document}